\documentclass[11pt]{article}

\usepackage{amssymb,amsmath,bm}
\usepackage{amsthm}

\usepackage{textcomp}
\usepackage{enumerate}      
\usepackage{graphicx}        
\usepackage{caption}

\usepackage{mathrsfs}

\usepackage{url}

\renewcommand{\setminus}{{\smallsetminus}}

\newcommand{\cp}[1]{\vcenter{\hbox{#1}}}

\usepackage{amssymb,amsmath,bm}
\usepackage{amsthm}
\usepackage{hyperref}
\usepackage{mathrsfs}
\usepackage{textcomp}
\usepackage{enumerate}
\usepackage{graphicx}
\usepackage{tikz}
\usepackage{verbatim}

\newtheorem{theorem}{Theorem}[section]
\newtheorem{lemma}[theorem]{Lemma}
\newtheorem{proposition}[theorem]{Proposition}
\newtheorem{definition}[theorem]{Definition}
\newtheorem{corollary}[theorem]{Corollary}
\newtheorem{conjecture}[theorem]{Conjecture}

\theoremstyle{remark}
\newtheorem{remark}[theorem]{Remark}

\theoremstyle{remark}

\numberwithin{equation}{section}

\begin{document}
\title{\bf Relative Reshetikhin-Turaev invariants, hyperbolic cone metrics and discrete Fourier transforms II}

\author{Ka Ho Wong and Tian Yang}
\date{}
\maketitle

\begin{abstract} We prove the Volume Conjecture for the relative  Reshetikhin-Turaev invariants proposed in \cite{WY2} for all pairs $(M,K)$ such that $M\setminus K$ is homeomorphic to the complement of the figure-$8$ knot in $S^3$ with almost all possible cone angles.
\end{abstract}

\section{Introduction}

In the first paper \cite{WY2} of this series, we proposed the Volume Conjecture for the relative  Reshetikhin-Turaev invariants of a closed oriented $3$-manifold with a colored framed link inside it whose asymptotic behavior is related to the volume and the Chern-Simons invariant of the hyperbolic cone metric on the manifold with singular locus the link and cone angles determined by the coloring. 

\begin{conjecture}[\cite{WY2}]\label{VC} Let $M$ be a closed oriented $3$-manifold and let $L$ be a framed hyperbolic link in $M$ with $n$ components. For an odd integer $r\geqslant 3,$ let  $\mathbf m=(m_1,\dots,m_n)$ and let $\mathrm{RT}_r(M,L,\mathbf m)$ be the $r$-th relative Reshetikhin-Turaev invariant of $M$ with $L$ colored by $\mathbf m$ and evaluated at the root of unity $q=e^{\frac{2\pi\sqrt{-1}}{r}}.$ 
For a sequence $\mathbf m^{(r)}=(m^{(r)}_1,\dots,m^{(r)}_n),$ let 
$$\theta_k=\bigg|2\pi- \lim_{r\to\infty}\frac{4\pi m^{(r)}_k}{r}\bigg|$$
and let $\theta=(\theta_1,\dots, \theta_n).$
If $M_{L_\theta}$ is a hyperbolic cone manifold consisting of $M$ and a hyperbolic cone metric on $M$ with singular locus $L$ and cone angles $\theta,$ then 
$$\lim_{r\to\infty} \frac{4\pi }{r}\log \mathrm{RT}_r(M,L,\mathbf m^{(r)})=\mathrm{Vol}(M_{L_\theta})+\sqrt{-1}\mathrm{CS}(M_{L_\theta})\quad\quad\text{mod } \sqrt{-1}\pi^2\mathbb Z,$$
where $r$ varies over all positive odd integers.
\end{conjecture}

In the same paper \cite{WY2}, we also proved Conjecture \ref{VC} in the case that the ambient $3$-manifold is obtained by doing an integral surgery along some components of a fundamental shadow link and the complement of the link in the ambient manifold is homeomorphic to the fundamental shadow link complement, for sufficiently small cone angles. A result of Costantino and Thurston shows that all the closed oriented $3$-manifolds  can be obtained by doing an integral Dehn filling along a suitable fundamental shadow link complement. On the other hand, it is expected that hyperbolic cone metrics interpolate the complete cusped hyperbolic metric on the $3$-manifold with toroidal boundary and the smooth hyperbolic metric on the Dehn-filled $3$-manifold, corresponding to the colors running from $\frac{r-1}{2}$ to $0$ or $r-2.$ Therefore, if one can push the cone angles in Theorem \ref{main} from sufficiently small all the way up to $2\pi,$ then one proves the Volume Conjecture for the Reshetikhin-Turaev invariants of closed oriented hyperbolic $3$-manifolds proposed by Chen and the second author\,\cite{CY}. This thus suggests a possible approach of solving Chen-Yang's Volume Conjecture.  

The main result  of this paper proves Conjecture \ref{VC} for all pairs $(M,K)$ such that $M\setminus K$ is homeomorphic to the figure-$8$ knot complement in $S^3$ with almost all possible cone angles, showing the plausibility of this new approach.

\begin{theorem}\label{main} Conjecture \ref{VC} is true for all pairs $(M,L)$ such that $M\setminus K$ is homeomorphic to the figure-$8$ knot complement $S^3\setminus K_{4_1}$ and for all cone angles $\theta$ such that
$$\mathrm{Vol}(M_{K_\theta})>\frac{\mathrm{Vol}(S^3 \setminus K_{4_1})}{2}.$$
 \end{theorem}

We note that if $M\setminus K$ is homeomorphic to $S^3\setminus K_{4_1},$ then $M$ is obtained from $S^3$ by doing a $\frac{p}{q}$-surgery along $K_{4_1}.$ In Proposition \ref{small}, we show that if $(p,q)\neq(\pm 1, 0),$ $(0,\pm 1),$ $(\pm 1,\pm 1),$ $(\pm 2,\pm 1),$ $(\pm 3,\pm 1),$ $(\pm 4,\pm 1)$ and $(\pm 5,\pm 1),$ then any $\theta$ less than or equal to $2\pi$ satisfies the condition of Theorem \ref{main}; and if $(p,q)$ is one of those sporadic cases except $(\pm 1,0),$ then any $\theta$ less than or equal to $\pi$ satisfies the condition of Theorem \ref{main}.
\\

\noindent\textbf{Outline of the proof.} The proof follows the guideline of Ohtsuki's method. In Proposition \ref{formula}, we compute the relative Reshetikhin-Turaev invariants of $(M,K)$ and write them as a sum of values of holomorphic functions $f^\pm_r$ at integral points. The functions $f^\pm_r$ comes from Faddeev's quantum dilogarithm function. Using Poisson summation formula, we in Proposition \ref{Poisson} write the invariants as a sum of the Fourier coefficients of $f_r,$ and in Propositions \ref{0}, \ref{n} and  \ref{other'} we simplify those Fourier coefficients by doing some preliminary estimate. In Proposition \ref{Vol} we show that the critical value of the functions in the two leading Fourier coefficients $\hat f^\pm_r(0,\dots, 0)$ have real part the  volume and imaginary part the Chern-Simons invariant of $M_{K_\theta}.$ The key observation is Lemma \ref{=}  that the system of critical point equations is equivalent to the system of hyperbolic gluing equations (consisting of an edge equation and a $\frac{p}{q}$ Dehn-filling equation with cone angle $\theta$) for a particular ideal triangulation of the figure-$8$ knot complement. In Section \ref{leading'} we verify the conditions for the saddle point approximation showing that the growth rate of the leading Fourier coefficients are the critical values; and in Section \ref{other}, we estimate the other Fourier coefficients showing that they are neglectable.
\\

\noindent\textbf{Acknowledgments.}  The authors would like to thank Feng Luo and Hongbin Sun for helpful discussions. The second author is partially supported by NSF Grant DMS-1812008.


\section{Preliminaries}

For the readers' convenience, we recall the relative Reshetikhin-Turaev invariants, hyperbolic cone metrics and the classical and quantum dilogarithm functions  respectively in Sections \ref{RRT}, \ref{CCS}, \ref{DL} and \ref{QDL}.

\subsection{Relative Reshetikhin-Turaev invariants}\label{RRT}

In this article we will follow the skein theoretical approach of the relative Reshetikhin-Turaev invariants\,\cite{BHMV, Li} and focus on the root of unity  $q=e^{\frac{2\pi\sqrt{-1}}{r}}$ for odd integers $r\geqslant 3.$

A framed link in an oriented $3$-manifold $M$ is a smooth embedding $L$ of a disjoint union of finitely many thickened circles $\mathrm S^1\times [0,\epsilon],$ for some $\epsilon>0,$ into $M.$ The Kauffman bracket skein module $\mathrm K_r(M)$ of $M$ is the $\mathbb C$-module generated by the isotopic classes of framed links in $M$  modulo the follow two relations: 

\begin{enumerate}[(1)]
\item  \emph{Kauffman Bracket Skein Relation:} \ $\cp{\includegraphics[width=1cm]{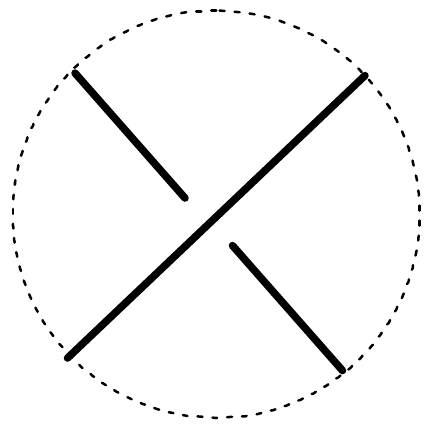}}\ =\ e^{\frac{\pi\sqrt{-1}}{r}}\ \cp{\includegraphics[width=1cm]{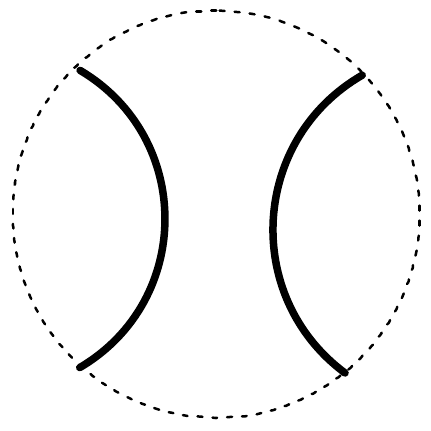}}\  +\ e^{-\frac{\pi\sqrt{-1}}{r}}\ \cp{\includegraphics[width=1cm]{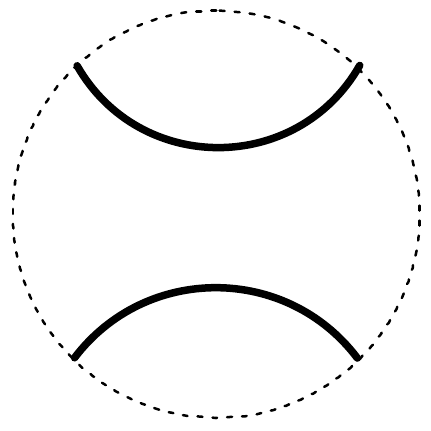}}.$ 

\item \emph{Framing Relation:} \ $L \cup \cp{\includegraphics[width=0.8cm]{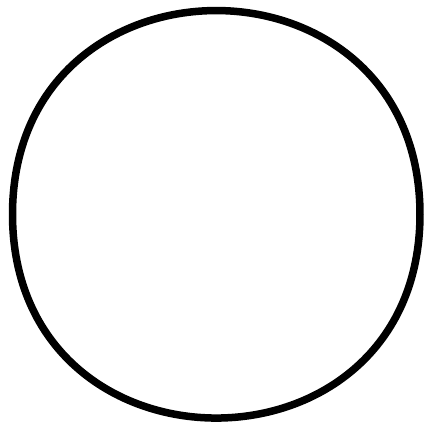}}=(-e^{\frac{2\pi\sqrt{-1}}{r}}-e^{-\frac{2\pi\sqrt{-1}}{r}})\ L.$ 
\end{enumerate}

There is a canonical isomorphism 
$$\langle\ \rangle:\mathrm K_r(\mathrm S^3)\to\mathbb C$$
defined by sending the empty link to $1.$ The image $\langle L\rangle$ of the framed link $L$ is called the Kauffman bracket of $L.$

Let $\mathrm K_r(A\times [0,1])$ be the Kauffman bracket skein module of the product of an annulus $A$ with a closed interval. For any link diagram $D$ in $\mathbb R^2$ with $k$ ordered components and $b_1, \dots, b_k\in \mathrm K_r(A\times [0,1]),$ let 
$$\langle b_1,\dots, b_k\rangle_D$$
be the complex number obtained by cabling $b_1,\dots, b_k$ along the components of $D$ considered as a element of $K_r(\mathrm S^3)$ then taking the Kauffman bracket $\langle\ \rangle.$

On $\mathrm K_r(A\times [0,1])$ there is a commutative multiplication induced by the juxtaposition of annuli, making it a $\mathbb C$-algebra; and as a $\mathbb C$-algebra $\mathrm K_r(A\times [0,1])  \cong \mathbb C[z],$ where $z$ is the core curve of $A.$ For an integer $n\geqslant 0,$ let $e_n(z)$ be the $n$-th Chebyshev polynomial defined recursively by
$e_0(z)=1,$ $e_1(z)=z$ and $e_n(z)=ze_{n-1}(z)-e_{n-2}(z).$ Let $\mathrm{I}_r=\{0,\dots,r-2\}$ be the set of integers in between $0$ and $r-2.$ Then the Kirby coloring $\Omega_r\in\mathrm K_r(A\times [0,1])$ is defined by 
$$\Omega_r=\mu_r\sum_{n\in \mathrm{I}_r}(-1)^n[n+1]e_{n},$$
where
 $$\mu_r=\frac{\sqrt2\sin\frac{2\pi}{r}}{\sqrt r}$$ 
 and $[n]$ is the quantum integer defined by
$$[n]=\frac{e^{\frac{2n\pi\sqrt{-1}}{r}}-e^{-\frac{2n\pi\sqrt{-1}}{r}}}{e^{\frac{2\pi\sqrt{-1}}{r}}-e^{-\frac{2\pi\sqrt{-1}}{r}}}.$$

Let $M$ be a closed oriented $3$-manifold and let $L$ be a framed link in $M$ with $n$ components. Suppose $M$ is obtained from $S^3$ by doing a surgery along a framed link $L',$ $D(L')$ is a standard diagram of $L'$ (ie, the blackboard framing of $D(L')$ coincides with the framing of $L'$). Then $L$ adds extra components to $D(L')$ forming a linking diagram $D(L\cup L')$ with $D(L)$ and $D(L')$ linking in possibly a complicated way. Let
$U_+$ be the diagram of the unknot with framing $1,$ $\sigma(L')$ be the signature of the linking matrix of $L'$ and $\mathbf m=(m_1,\dots,m_n)$ be a multi-elements of $I_r.$ Then the $r$-th \emph{relative Reshetikhin-Turaev invariant of $M$ with $L$ colored by $\mathbf m$} is defined as
\begin{equation}\label{RT}
\mathrm{RT}_r(M,L,\mathbf m)=\mu_r \langle e_{m_1},\dots,e_{m_n}, \Omega_r, \dots, \Omega_r\rangle_{D(L\cup L')}\langle \Omega_r\rangle _{U_+}^{-\sigma(L')}.
\end{equation}

Note that if $L=\emptyset$ or $m_1=\dots=m_n=0,$ then $\mathrm{RT}_r(M,L,\mathbf m)=\mathrm{RT}_r(M),$ the $r$-th Reshetikhin-Turaev invariant of $M;$ and if $M=S^3,$ then $\mathrm{RT}_r(M,L,\mathbf m)=\mu_r\mathrm J_{\mathbf m, L}(q^2),$ the value of the $\mathbf m$-th unnormalized colored Jones polynomial of $L$ at $t=q^2.$

\subsection{Hyperbolic cone manifolds}\label{CCS}

According to \cite{CHK}, a $3$-dimensional \emph{hyperbolic cone-manifold} is a $3$-manifold $M,$ which can be triangulated so that the link of each simplex is piecewise linear homeomorphic to a standard sphere and $M$ is equipped with a complete path metric such that the restriction of the metric to each simplex is isometric to a hyperbolic geodesic simplex. The \emph{singular locus} $L$ of a cone-manifold $M$ consists of the points with no neighborhood isometric to a ball in a Riemannian manifold. It follows that
\begin{enumerate}[(1)]
\item $L$ is a link in $M$ such that each component is a closed geodesic.

\item At each point of $L$ there is a \emph{cone angle} $\theta$ which is the sum of dihedral angles of $3$-simplices containing the point.

\item The restriction of the metric on $M\setminus L$ is a smooth hyperbolic metric, but is incomplete if $L\neq \emptyset.$
\end{enumerate}

Hodgson-Kerckhoff\,\cite{HK} proved that hyperbolic cone metrics on $M$ with singular locus $L$ are locally parametrized by the cone angles provided all the cone angles are less than or equal to $2\pi,$ and Kojima\,\cite{Ko} proved that hyperbolic cone manifolds $(M,L)$ are globally rigid provided all the cone angles are less than or equal to $\pi.$ It is expected to be  globally rigid if all the cone angles are less than or equal to $2\pi.$

Given a $3$-manifold $N$ with boundary a union of tori $T_1,\dots,T_n,$ a choice of generators $(u_i,v_i)$ for each $\pi_1(T_i)$ and pairs of relatively prime integers $(p_i, q_i),$ one can do the $(\frac{p_1}{q_1},\dots,\frac{p_k}{q_k})$-Dehn filling on $N$ by attaching a solid torus to each $T_i$ so that $p_iu_i+q_iv_i$ bounds a disk.  If $\mathrm H(u_i)$ and  $\mathrm H(v_i)$ are respectively the logarithmic holonomy for $u_i$ and $v_i,$ then a solution to  
\begin{equation}\label{DF}
p_i\mathrm H(u_i)+q_i\mathrm H(v_i)=\sqrt{-1}\theta_i
\end{equation}
 near the complete structure gives a cone-manifold structure on the resulting manifold $M$ with the cone angle $\theta_i$ along the core curve $L_i$ of the solid torus attached to $T_i;$ it is a smooth structure if $\theta_1=\dots=\theta_n= 2\pi.$

In this setting, the Chern-Simons invariant for a hyperbolic cone manifold $(M,L)$ can be defined by using the Neumann-Zagier potential function\,\cite{NZ}. To do this, we need a framing on each component, namely, a choice of a curve $\gamma_i$ on $T_i$ that is isotopic to the core curve $L_i$ of the solid torus attached to $T_i.$ We choose the orientation of $\gamma_i$ so that $(p_iu_i+q_iv_i)\cdot \gamma_i=1.$ Then we consider the following function 
$$\frac{\Phi(\mathrm{H}(u_1),\dots,\mathrm{H}(u_n))}{\sqrt{-1}}-\sum_{i=1}^n\frac{\mathrm H(u_i)\mathrm H(v_i)}{4\sqrt{-1}}+\sum_{i=1}^n\frac{\theta_i\mathrm H(\gamma_i)}{4},$$
where $\Phi$ is the Neumann-Zagier potential function (see \cite{NZ}) defined on the deformation space of hyperbolic structures on $M\setminus L$ parametrized by the holonomy of the meridians $\{\mathrm H(u_i)\},$ characterized by 
\begin{equation}\label{char}
\left \{\begin{array}{l}
\frac{\partial \Phi(\mathrm{H}(u_1),\dots,\mathrm{H}(u_n))}{\partial \mathrm{H}(u_i)}=\frac{\mathrm H(v_i)}{2},\\
\\
\Phi(0,\dots,0)=\sqrt{-1}\bigg(\mathrm{Vol}(M\setminus L)+\sqrt{-1}\mathrm{CS}(M\setminus L)\bigg)\quad\quad\mathrm{mod}\ \pi^2\mathbb Z,
\end{array}\right.
\end{equation} 
where $M\setminus L$ is with the complete hyperbolic metric. Another important feature of $\Phi$ is that it is even in each of its variables $\mathrm H(u_i).$

Following the argument in \cite[Sections 4 \& 5]{NZ}, one can prove that if  the cone angles of components of $L$ are $\theta_1,\dots,\theta_n,$ then
\begin{equation}\label{VOL}
\mathrm{Vol}(M_{L_\theta})=\mathrm{Re}\bigg(\frac{\Phi(\mathrm{H}(u_1),\dots,\mathrm{H}(u_n))}{\sqrt{-1}}-\sum_{i=1}^n\frac{\mathrm H(u_i)\mathrm H(v_i)}{4\sqrt{-1}}+\sum_{i=1}^n\frac{\theta_i\mathrm H(\gamma_i)}{4}\bigg).
\end{equation}
Indeed, in this case, one can replace the $2\pi$ in Equations (33) (34) and (35) of \cite{NZ} by $\theta_i,$ and as a consequence can replace the $\frac{\pi}{2}$ in Equations (45), (46) and (48) by $\frac{\theta_i}{4},$ proving the result.

In \cite{Y}, Yoshida proved that when $\theta_1=\dots=\theta_n=2\pi,$
$$\mathrm{Vol}(M)+\sqrt{-1}\mathrm{CS}(M)=\frac{\Phi(\mathrm{H}(u_1),\dots,\mathrm{H}(u_n))}{\sqrt{-1}}-\sum_{i=1}^n\frac{\mathrm H(u_i)\mathrm H(v_i)}{4\sqrt{-1}}+\sum_{i=1}^n\frac{\theta_i\mathrm H(\gamma_i)}{4}\quad\quad\text{mod }\sqrt{-1}\pi^2\mathbb Z.$$

Therefore, we can make the following 

\begin{definition}\label{CS} The \emph{Chern-Simons invariant} of a hyperbolic cone manifold $M_{L_\theta}$ with a choice of the framing $(\gamma_1,\dots,\gamma_n)$ 
is defined as
$$\mathrm{CS}(M_{L_\theta})=\mathrm{Im}\bigg(\frac{\Phi(\mathrm{H}(u_1),\dots,\mathrm{H}(u_n))}{\sqrt{-1}}-\sum_{i=1}^n\frac{\mathrm H(u_i)\mathrm H(v_i)}{4\sqrt{-1}}+\sum_{i=1}^n\frac{\theta_i\mathrm H(\gamma_i)}{4}\bigg)\quad\quad\text{mod }\pi^2\mathbb Z.$$
\end{definition}

Then together with (\ref{VOL}), we have
\begin{equation}\label{VCS}
\mathrm{Vol}(M_{L_\theta})+\sqrt{-1}\mathrm{CS}(M_{L_\theta})=\frac{\Phi(\mathrm{H}(u_1),\dots,\mathrm{H}(u_n))}{\sqrt{-1}}-\sum_{i=1}^n\frac{\mathrm H(u_i)\mathrm H(v_i)}{4\sqrt{-1}}+\sum_{i=1}^n\frac{\theta_i\mathrm H(\gamma_i)}{4}\quad\quad\text{mod }\sqrt{-1}\pi^2\mathbb Z.
\end{equation}

\begin{remark} It is an interesting question to find a direct geometric definition of the Chern-Simons invariants for hyperbolic cone manifolds.
\end{remark}



\subsection{Dilogarithm and Lobachevsky functions}\label{DL}

Let $\log:\mathbb C\setminus (-\infty, 0]\to\mathbb C$ be the standard logarithm function defined by
$$\log z=\log|z|+i\arg z$$
with $-\pi<\arg z<\pi.$
 
The dilogarithm function $\mathrm{Li}_2: \mathbb C\setminus (1,\infty)\to\mathbb C$ is defined by
$$\mathrm{Li}_2(z)=-\int_0^z\frac{\log (1-u)}{u}du,$$
which is holomorphic in $\mathbb C\setminus [1,\infty)$ and continuous in $\mathbb C\setminus (1,\infty).$

The dilogarithm function satisfies the follow properties (see eg. Zagier\,\cite{Z})
\begin{equation}\label{Li2}
\mathrm{Li}_2\Big(\frac{1}{z}\Big)=-\mathrm{Li}_2(z)-\frac{\pi^2}{6}-\frac{1}{2}\big(\log(-z)\big)^2.
\end{equation} 
In the unit disk $\big\{z\in\mathbb C\,\big|\,|z|<1\big\},$ 
\begin{equation}\label{Li1}
\mathrm{Li}_2(z)=\sum_{n=1}^\infty\frac{z^n}{n^2},
\end{equation}
and on the unit circle $\big\{ z=e^{2i\theta}\,\big|\,0 \leqslant \theta\leqslant\pi\big\},$ 
\begin{equation}\label{dilogLob}
\mathrm{Li}_2(e^{2i\theta})=\frac{\pi^2}{6}+\theta(\theta-\pi)+2i\Lambda(\theta),
\end{equation}
where  $\Lambda:\mathbb R\to\mathbb R$  is the Lobachevsky function (see eg. Thurston's notes\,\cite[Chapter 7]{T}) defined by
$$\Lambda(\theta)=-\int_0^\theta\log|2\sin t|dt.$$
The Lobachevsky function  is an odd function of period $\pi.$ 


\subsection{Quantum dilogarithm functions}\label{QDL}
We will consider the following variant of Faddeev's quantum dilogarithm functions\,\cite{F, FKV}. 
Let $r\geqslant 3$ be an odd integer. Then the following contour integral
\begin{equation}
\varphi_r(z)=\frac{4\pi i}{r}\int_{\Omega}\frac{e^{(2z-\pi)x}}{4x \sinh (\pi x)\sinh (\frac{2\pi x}{r})}\ dx
\end{equation}
defines a holomorphic function on the domain $$\Big\{z\in \mathbb C \ \Big|\ -\frac{\pi}{r}<\mathrm{Re}z <\pi+\frac{\pi}{r}\Big\},$$  
  where the contour is
$$\Omega=\big(-\infty, -\epsilon\big]\cup \big\{z\in \mathbb C\ \big||z|=\epsilon, \mathrm{Im}z>0\big\}\cup \big[\epsilon,\infty\big),$$
for some $\epsilon\in(0,1).$
Note that the integrand has poles at $ni,$ $n\in\mathbb Z,$ and the choice of  $\Omega$ is to avoid the pole at $0.$
\\

The function $\varphi_r(z)$ satisfies the following fundamental property.
\begin{lemma}
\begin{enumerate}[(1)]
\item For $z\in\mathbb C$ with  $0<\mathrm{Re}z<\pi,$
\begin{equation}\label{fund}
1-e^{2 i z}=e^{\frac{r}{4\pi i}\Big(\varphi_r\big(z-\frac{\pi}{r}\big)-\varphi_r\big(z+\frac{\pi}{r}\big)\Big)}
 \end{equation}
 
 \item For $z\in\mathbb C$ with  $-\frac{\pi}{r}<\mathrm{Re}z<\frac{\pi}{r},$
 \begin{equation}\label{f2}
1+e^{riz}=e^{\frac{r}{4\pi i}\Big(\varphi_r(z)-\varphi_r\big(z+\pi\big)\Big)}.
\end{equation}
\end{enumerate}
\end{lemma}

Using (\ref{fund}) and (\ref{f2}), for $z\in\mathbb C$ with $\pi+\frac{2(n-1)\pi}{r}< \mathrm{Re}z< \pi+\frac{2n\pi}{r},$ we can define $\varphi_r(z)$  inductively by the relation
\begin{equation}\label{extension}
\prod_{k=1}^n\Big(1-e^{2 i \big(z-\frac{(2k-1)\pi}{r}\big)}\Big)=e^{\frac{r}{4\pi i}\Big(\varphi_r\big(z-\frac{2n\pi}{r}\big)-\varphi_r(z)\Big)},
\end{equation}
extending $\varphi_r(z)$ to a meromorphic function on $\mathbb C.$  The poles of $\varphi_r(z)$ have the form $(a+1)\pi+\frac{b\pi}{r}$ or $-a\pi-\frac{b\pi}{r}$ for all nonnegative integer $a$ and positive odd integer $b.$

Let $q=e^{\frac{2\pi i}{r}},$
and let $$(q)_n=\prod_{k=1}^n(1-q^{2k}).$$

\begin{lemma}\label{fact}
\begin{enumerate}[(1)]
\item For $0\leqslant n \leqslant r-2,$
\begin{equation}
(q)_n=e^{\frac{r}{4\pi i}\Big(\varphi_r\big(\frac{\pi}{r}\big)-\varphi_r\big(\frac{2\pi n}{r}+\frac{\pi}{r}\big)\Big)}.
\end{equation}
\item For $\frac{r-1}{2}\leqslant n \leqslant r-2,$
\begin{equation} \label{move}
(q)_n=2e^{\frac{r}{4\pi i}\Big(\varphi_r\big(\frac{\pi}{r}\big)-\varphi_r\big(\frac{2\pi n}{r}+\frac{\pi}{r}-\pi\big)\Big)}.
\end{equation}
\end{enumerate}
\end{lemma}

Since 
$$\{n\}!=(-1)^nq^{-\frac{n(n+1)}{2}}(q)_n,$$
as a consequence of Lemma \ref{fact}, we have

\begin{lemma}\label{factorial}
\begin{enumerate}[(1)]
\item For $0\leqslant n \leqslant r-2,$
\begin{equation}
\{n\}!=e^{\frac{r}{4\pi i}\Big(-2\pi\big(\frac{2\pi n}{r}\big)+\big(\frac{2\pi}{r}\big)^2(n^2+n)+\varphi_r\big(\frac{\pi}{r}\big)-\varphi_r\big(\frac{2\pi n}{r}+\frac{\pi}{r}\big)\Big)}.
\end{equation}
\item For $\frac{r-1}{2}\leqslant n \leqslant r-2,$
\begin{equation} \label{move}
\{n\}!=2e^{\frac{r}{4\pi i}\Big(-2\pi\big(\frac{2\pi n}{r}\big)+\big(\frac{2\pi }{r}\big)^2(n^2+n)+\varphi_r\big(\frac{\pi}{r}\big)-\varphi_r\big(\frac{2\pi n}{r}+\frac{\pi}{r}-\pi\big)\Big)}.
\end{equation}
\end{enumerate}
\end{lemma}

We consider (\ref{move}) because there are poles in $(\pi,2\pi),$ so we move everything to $(0,\pi)$ to avoid the poles.

The function $\varphi_r(z)$ and the dilogarithm function are closely related as follows.

\begin{lemma}\label{converge}  \begin{enumerate}[(1)]
\item For every $z$ with $0<\mathrm{Re}z<\pi,$ 
\begin{equation}\label{conv}
\varphi_r(z)=\mathrm{Li}_2(e^{2iz})+\frac{2\pi^2e^{2iz}}{3(1-e^{2iz})}\frac{1}{r^2}+O\Big(\frac{1}{r^4}\Big).
\end{equation}
\item For every $z$ with $0<\mathrm{Re}z<\pi,$ 
\begin{equation}\label{conv}
\varphi_r'(z)=-2i\log(1-e^{2iz})+O\Big(\frac{1}{r^2}\Big).
\end{equation}
\end{enumerate}\end{lemma}


\section{Computation of the relative Reshetikhin-Turaev invariants}

Let $(M,K)$ be a pair such that $M\setminus K$ is homeomorphic to $S^3\setminus K_{4_1}.$ Then $M$ is obtained from $S^3$ by doing a $\frac{p}{q}$ Dehn-filling along $K_{4_1}$ and $K$ is isotopic to a curve on the boundary of the tubular neighborhood of $K_{4_1}$ that intersects the $(p,q)$-curve of the boundary at exactly one point.  By eg. \cite[p.273]{R}, $M$ can also be obtained by doing a surgery along a framed link $L$ of $k$ components with framings $a_1,\dots, a_k$ coming from the continued fraction expansion
$$\frac{p}{q}=a_k-\frac{1}{a_{k-1}-\frac{1}{\cdots-\frac{1}{a_1}}}$$ of $\frac{p}{q},$ and $K$ is a framed trivial loop isotopic to the meridian of the last loop (see Figure \ref{link}).

\begin{figure}[htbp]
\centering
\includegraphics[scale=0.3]{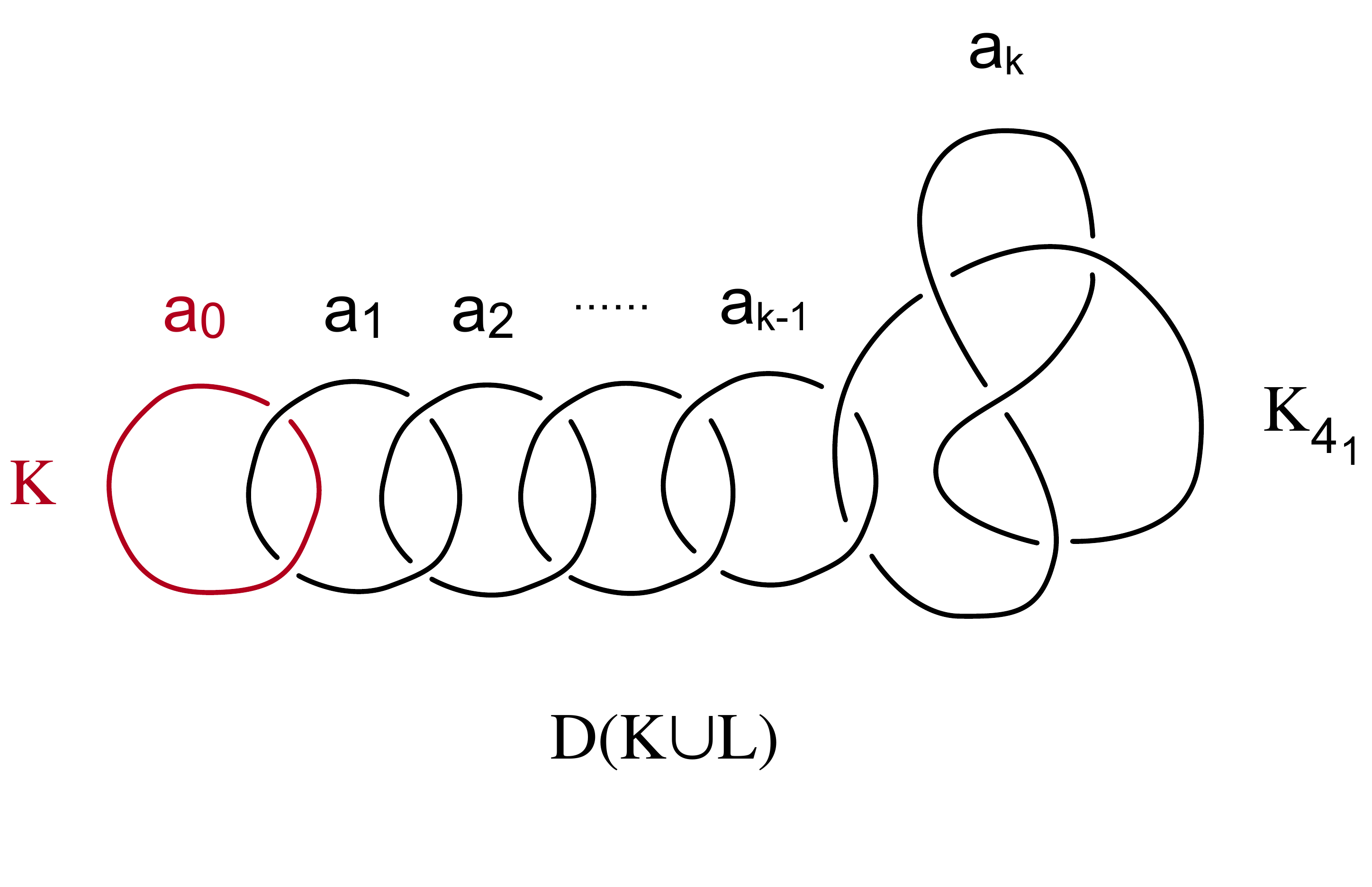}
\caption{The Kirby diagram of $(M,K)$}
\label{link}
\end{figure}

\begin{proposition}\label{formula}  For an odd integer $r\geqslant 3,$ the $r$-th relative Reshetikin-Turaev invariant $\mathrm{RT}_r(M,K,m_0)$ of the triple $(M,K,m_0)$ at the root $q=e^{\frac{2\pi i}{r}}$ can be computed as
\begin{equation*}
\begin{split} 
\mathrm{RT}_r(M,K,m_0)=\kappa_r\sum_{m_1,\dots, m_k=-\frac{r-2}{2}}^{\frac{r-2}{2}}\sum_{m=\max\{-m_k,m_k\}}^{\frac{r-2}{2}}\Big(g^+_r(m_1,\dots,m_{k}, m)+g^-_r(m_1,\dots,m_{k}, m)\Big),
\end{split}
\end{equation*}
where 
$$\kappa_r=\frac{2^{k-3}}{{\sqrt {r^{k+1}}}}\Big(\sin\frac{2\pi}{r}\Big)^{k-1}e^{\big(3\sum_{i=0}^ka_i+\sigma(L)+2k-2\big)\frac{r\pi i}{4}+\big(-\sum_{i=0}^ka_i(1+\frac{1}{r})+\frac{3\sigma(L)}{r}+\frac{\sigma(L)}{4}\big)\pi i}$$
and
$$g^\pm_r(m_1,\dots,m_{k}, m)=\epsilon\Big(\frac{2\pi m_k}{r},\frac{2\pi m}{r}\Big)e^{-\frac{2\pi m_k i}{r}+\frac{r}{4\pi i}V^\pm_r
\Big(\frac{2\pi m_1}{r},\dots,\frac{2\pi m_k}{r},\frac{2\pi m}{r}\Big)}$$
with $\epsilon$ and $V_r$ defined as follows. Let 
$$x_0=\frac{2\pi}{r}\Big(\frac{r-2}{2}-m_0\Big)=\pi-\frac{2\pi}{r}-\frac{2\pi m_0}{r}.$$
\begin{enumerate}[(1)]
\item If both $0<y\pm x_k< \pi,$ then $\epsilon(x_k,y)=2$ and
$$V^\pm_r(x_1,\dots,x_k,y)=\pm 2 x_0 x_1-\sum_{i=0}^{k}a_ix_i^2-\sum_{i=1}^{k-1}2x_ix_{i+1}-2\pi x_k+4x_ky-\varphi_r\Big(\pi-y-x_k-\frac{\pi}{r}\Big)+\varphi_r\Big(y-x_k+\frac{\pi}{r}\Big).$$
\item If $0< y+x_k<\pi$ and $\pi < y-x_k < 2\pi,$ then $\epsilon(x_k,y)=1$ and
$$V^\pm_r(x_1,\dots,x_k,y)=\pm 2 x_0 x_1-\sum_{i=0}^{k}a_ix_i^2-\sum_{i=1}^{k-1}2x_ix_{i+1}-2\pi x_k+4x_ky-\varphi_r\Big(\pi-y-x_k-\frac{\pi}{r}\Big)+\varphi_r\Big(y-x_k-\pi+\frac{\pi}{r}\Big).$$
\item If $\pi < y+x_k < 2\pi$ and $0 < y-x_k < \pi,$ then $\epsilon(x_k,y)=1$ and
$$V^\pm_r(x_1,\dots,x_k,y)=\pm 2 x_0 x_1-\sum_{i=0}^{k}a_ix_i^2-\sum_{i=1}^{k-1}2x_ix_{i+1}-2\pi x_k+4x_ky-\varphi_r\Big(2\pi-y-x_k-\frac{\pi}{r}\Big)+\varphi_r\Big(y-x_k+\frac{\pi}{r}\Big).$$
\end{enumerate}
\end{proposition}

\begin{proof} A  direct computation shows that
$$\langle\mu_r \omega_r\rangle _{U_+}=e^{\big(-\frac{3}{r}-\frac{r+1}{4}\big)\pi i}.$$
Let 
$$\kappa_r'=\mu_r^{k+1}\langle\mu_r \omega_r\rangle _{U_+}^{-\sigma(L)}=\bigg(\frac{\sqrt2\sin\frac{2\pi}{r}}{\sqrt r}\bigg)^{k+1}e^{-\sigma(L)\big(-\frac{3}{r}-\frac{r+1}{4}\big)\pi i}.$$
Then by (\ref{RT}), we have
\begin{equation*}
\begin{split}
\mathrm{RT}_r(M)&=\kappa_r' \langle e_{m_0}, \omega_r,\dots,\omega_r\rangle_{D(K\cup L)}  \\
&=\kappa_r'\sum_{m_1,\dots, m_k=0}^{r-2}(-1)^{m_0+m_k+\sum_{i=0}^ka_im_i}q^{\sum_{i=0}^n\frac{a_im_i(m_i+2)}{2}}\prod_{i=0}^{k-1}[(m_i+1)(m_{i+1}+1)]\langle e_{m_k} \rangle_{D(K_{4_1})},
\end{split}
\end{equation*}
where the second equality comes from the fact that $e_n$ is an eigenvector of the positive and the negative twist operator of eigenvalue $(-1)^nq^{\pm\frac{n(n+2)}{2}}$, and is also an eigenvector of the circle operator $c(e_m)$ (defined by enclosing $e_n$ by $e_m$) of eigenvalue $(-1)^m\frac{[(m+1)(n+1)]}{[n+1]}.$
By Habiro's formula\,\cite{Ha} (see also \cite{Ma} for a skein theoretical computation) 
\begin{equation*}
\begin{split}
\langle e_n \rangle_{D(K_{4_1})}=\frac{(-1)^{n+1}}{\{1\}}\sum_{m=0}^{\min\{n,r-2-n\}}q^{-2(n+1)(m+\frac{1}{2})}\frac{(q)_{n+1+m}}{(q)_{n-m}},
\end{split}
\end{equation*}
where $\{n\}=q^{n}-q^{-n}$ and $(q)_n=\prod_{k=1}^n(1-q^{2k}).$

Then
\begin{equation}\label{sum}
\begin{split}
\mathrm{RT}_r(M)=&\frac{(-1)^{m_0+1}}{\{1\}}\kappa_r'\sum_{m_1,\dots, m_k=0}^{r-2}\sum_{m=0}^{\min\{m_k,r-2-m_k\}}(-1)^{\sum_{i=0}^ka_im_i}q^{\sum_{i=1}^n\frac{a_im_i(m_i+2)}{2}-2(m_k+1)(m+\frac{1}{2})}\\
&\quad\quad\quad\quad\quad\quad\quad\quad\quad\quad\quad\quad\prod_{i=0}^{k-1}[(m_i+1)(m_{i+1}+1)]\frac{(q)_{m_k+1+m}}{(q)_{m_k-m}}.\\
=&\frac{(-1)^{m_0+1}2^{k-1}}{\{1\}^2}\kappa_r'\sum_{m_1,\dots, m_k=0}^{r-2}\sum_{m=0}^{\min\{m_k,r-2-m_k\}}\Big(q^{(m_0+1)(m_1+1)}-q^{-(m_0+1)(m_1+1)}\Big)\\
&\quad\quad(-1)^{\sum_{i=0}^ka_im_i}q^{\sum_{i=0}^k\frac{a_im_i(m_i+2)}{2}+\sum_{i=1}^{k-1}(m_i+1)(m_{i+1}+1)-2(m_k+1)(m+\frac{1}{2})}\frac{(q)_{m_k+1+m}}{(q)_{m_k-m}},\\
\end{split}
\end{equation}
where the last equality comes from the following computation. For each $i\in\{1,\dots, k-1\}$ let $\epsilon_i\in\{0,1\},$ and let 
\begin{equation*}
\begin{split}
&S^{(\epsilon_1,\dots,\epsilon_{k-1})}(m_1,\dots,m_{k-1})\\
=&\frac{(-1)^{m_0+1}}{\{1\}^2}\kappa_r'\sum_{m_k=0}^{r-2}\sum_{m=0}^{\min\{m_k,r-2-m_k\}}\Big(q^{(m_0+1)(m_1+1)}-q^{-(m_0+1)(m_1+1)}\Big)\\
&(-1)^{\sum_{i=0}^ka_im_i+\sum_{i=1}^{k-1}\epsilon_i}q^{\sum_{i=0}^k\frac{a_im_i(m_i+2)}{2}+\sum_{i=1}^{k-1}(-1)^{\epsilon_i}(m_i+1)(m_{i+1}+1)-2(m_k+1)(m+\frac{1}{2})}\frac{(q)_{m_k+1+m}}{(q)_{m_k-m}}.
\end{split}
\end{equation*}
Then
$$\mathrm{RT}_r(M)=\sum_{m_1,\dots, m_{k-1}=0}^{r-2}\sum_{\epsilon_1,\dots,\epsilon_{k-1}\in\{0,1\}}S^{(\epsilon_1,\dots,\epsilon_{k-1})}(m_1,\dots,m_{k-1}).$$
For each $i\in\{1,\dots,, k-1\},$ a direct computation shows that
\begin{equation*}
\begin{split}
S^{(0,\dots,0,\epsilon_{i+1},\dots,\epsilon_{k-1})}&(m_1,\dots,m_{i+1},\dots,m_{k-1})\\
&=S^{(0,\dots,0,1,\epsilon_{i+1},\dots,\epsilon_{k-1})}(r-2-m_1,\dots, r-2-m_i,m_{i+1},\dots,m_{k-1}).
\end{split}
\end{equation*}
Then we have
$$\sum_{m_1,\dots, m_{k-1}=0}^{r-2}S^{(0,\dots,0,\epsilon_{i+1},\dots,\epsilon_{k-1})}(m_1,\dots,m_{k-1})=\sum_{m_1,\dots, m_{k-1}=0}^{r-2}S^{(0,\dots,0,1,\epsilon_{i+1},\dots,\epsilon_{k-1})}(m_1,\dots,m_{k-1}),$$
and hence
$$\mathrm{RT}_r(M)=2^{k-1}\sum_{m_1,\dots, m_{k-1}=0}^{r-2}S^{(0,\dots,0)}(m_1,\dots,m_{k-1}),$$
which proves (\ref{sum}).

Now let $m'=\frac{r-2}{2}-m,$ and for each $i\in\{0,1,\dots, k\}$ let $m_i'=\frac{r-2}{2}-m_i.$ Then 
\begin{equation*}
\begin{split}
\mathrm{RT}_r(M)={\kappa_r} \sum_{m_1',\dots, m_k'=-\frac{r-2}{2}}^{\frac{r-2}{2}}&\sum_{m'=\max\{-m_k',m_k'\}}^{\frac{r-2}{2}}\Big(q^{m_0'm_1'}+q^{-m_0'm_1'}\Big)\\
&(-1)^{m_k'}q^{\sum_{i=0}^k\frac{a_im_i'^2}{2}+\sum_{i=1}^{k-1}{m_i'm_{i+1}'}-2m_k'm'-{m_k'}}\frac{(q)_{r-1-m'-m_k'}}{(q)_{m'-m_k'}},\\
\end{split}
\end{equation*}
and by Lemma \ref{fact}, 
$$\mathrm{RT}_r(M)=\kappa_r\sum_{m_1',\dots, m_k'=-\frac{r-2}{2}}^{\frac{r-2}{2}}\sum_{m'=\max\{-m_k',m_k'\}}^{\frac{r-2}{2}}\Big(g^+_r(m_1',\dots,m_{k}', m')+g^-_r(m_1',\dots,m_{k}', m')\Big).$$
\end{proof}


\section{Poisson summation formula}

We notice that the summation in Proposition \ref{formula} is finite, and to use the Poisson summation formula, we need an infinite sum over integral points. To this end, we consider the following regions and a bump function over them.

Let $x_i=\frac{2\pi m_i}{r}$ for each $i\in\{1,\dots, k\},$ and let $y=\frac{2\pi m}{r}.$ For $\delta\geqslant 0,$ we let 
$$D_{\delta}=\Big\{(x,y)\in \mathbb R^2\ \Big|\ \delta < y+x <\frac{\pi}{2}-\delta, \delta<y-x<\frac{\pi}{2}-\delta\Big\},$$
$$D'_{\delta}=\Big\{(x,y)\in \mathbb R^2\ \Big|\ \delta< y+x <\frac{\pi}{2}-\delta, \pi+\delta< y-x < \frac{3\pi}{2}-\delta \Big\}$$
and
$$D''_{\delta}=\Big\{(x,y)\in \mathbb R^2\ \Big|\ \pi+\delta <  y+x<\frac{3\pi}{2}-\delta, \delta<y-x <\frac{\pi}{2}-\delta\Big\},$$
and let  $\mathcal D_\delta=D_\delta\cup D'_\delta\cup D''_\delta.$ If $\delta=0,$ we omit the subscript and write
$D=D_0,$ $D'=D'_0,$ $D''=D''_0$ and $\mathcal D=D\cup D'\cup D''.$

\begin{figure}[htbp]
\centering
\includegraphics[scale=0.7]{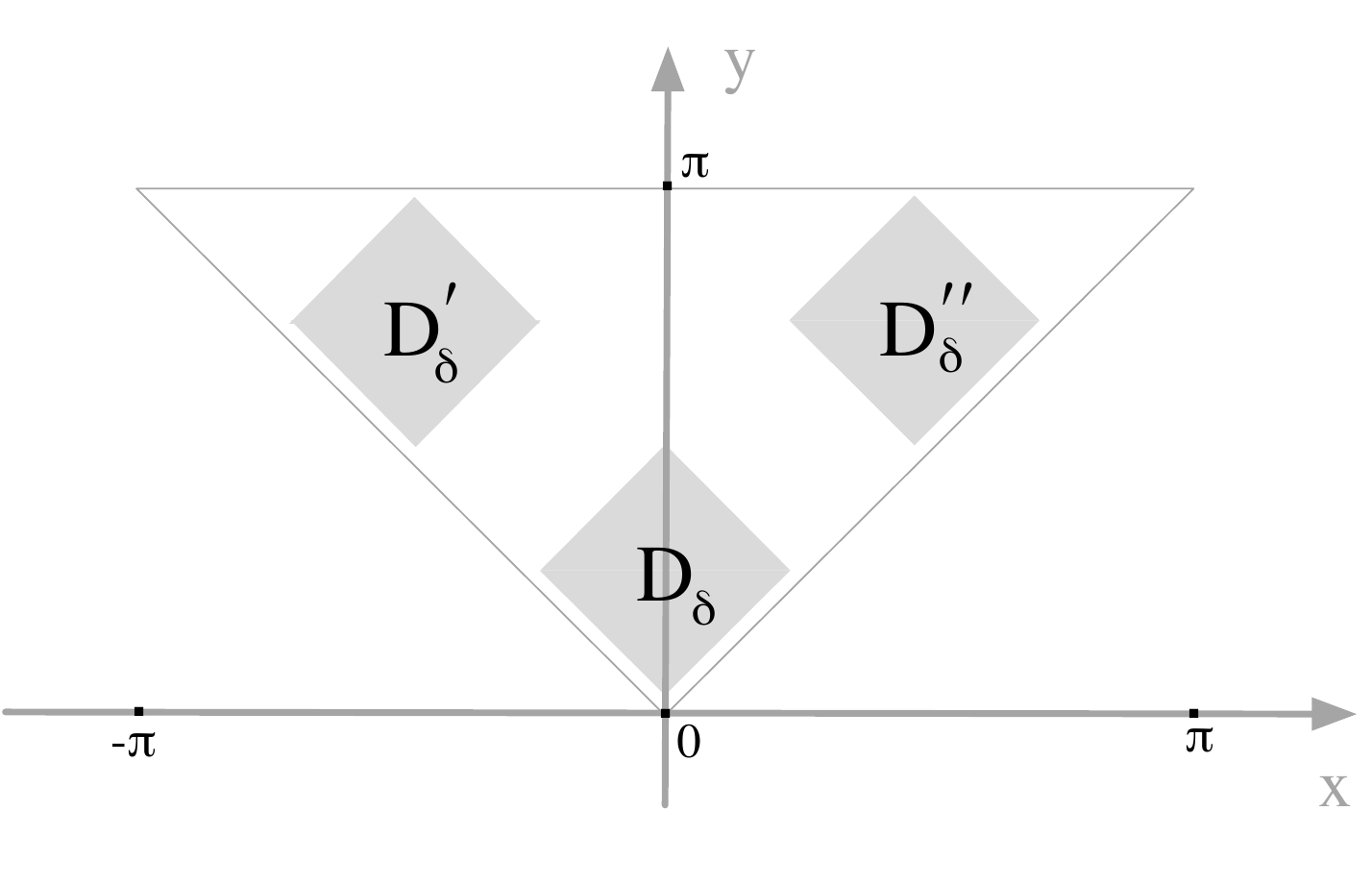}
\caption{Regions $D_\delta,$ $D_\delta'$ and $D_\delta''$}
\label{Delta} 
\end{figure}

For a sufficiently small $\delta>0,$ we consider a $C^{\infty}$-smooth bump function $\psi$ on $\mathbb R^{k+1}$ such that 
\begin{equation*}
\left \{\begin{array}{rl}
\psi(x_1,\dots,x_k,y)=1, & (x_1,\dots,x_k,y)\in[-\pi+\frac{2\pi}{r},\pi-\frac{2\pi}{r}]^{k-1}\times\overline{\mathcal D_{\frac{\delta}{2}}}\\
0<\psi(x_1,\dots,x_k,y)<1, & (x_1,\dots,x_k,y)\in(-\pi,\pi)^{k-1}\times\mathcal D\setminus [-\pi+\frac{2\pi}{r},\pi-\frac{2\pi}{r}]^{k-1}\times\overline{\mathcal D_{\frac{\delta}{2}}}\\
\psi(x_1,\dots,x_k,y)=0, & (x_1,\dots,x_k,y)\notin(-\pi,\pi)^{k-1}\times\mathcal D,\\
\end{array}\right.
\end{equation*}
and let 
$$f^\pm_r(m_1,\dots,m_k,m)=\psi\Big(\frac{2\pi m_1}{r},\dots,\frac{2\pi m_k}{r},\frac{2\pi m}{r}\Big)g^\pm_r(m_1,\dots, m_k,m).$$
Then by Proposition \ref{bound}, we have
\begin{equation}\label{halfinteger}
\mathrm{RT}_r(M)=\kappa_r \sum_{(m_1,\dots,m_k,m)\in(\mathbb Z+\frac{1}{2})^{k+1}}\Big(f^+_r(m_1,\dots,m_k,m)+f^-_r(m_1,\dots,m_k,m)\Big)+\text{error term}.
\end{equation}

Since $f^\pm_r$ is $C^{\infty}$-smooth and equals zero out of $D,$ it is in the Schwartz space on $\mathbb R^{k+1}.$ Recall that by the Poisson Summation Formula (see e.g. \cite[Theorem 3.1]{SS}), for any function $f$ in the Schwartz space on $\mathbb R^{k},$
$$\sum_{(m_1,\dots,m_k)\in\mathbb Z^{k}}f(m_1,\dots,m_k)=\sum_{(n_1,\dots,n_{k})\in\mathbb Z^{k}}\hat f(n_1,\dots,n_{k}),$$
where  $\hat f(n_1,\dots,n_{k})$ is the $(n_1,\dots,n_{k})$-th Fourier coefficient of $f$ defined by
$$\widehat {f^\pm}(n_1,\dots,n_{k})=\int_{\mathbb R^{k}}f^\pm(m_1,\dots, m_k)e^{\sum_{j=1}^k2\pi in_jm_j}dm_1\dots dm_k.$$
As a consequence, we have

\begin{proposition}\label{Poisson}
$$\mathrm{RT}_r(M)=\kappa_r \sum_{(n_1,\dots,n_{k}, n)\in\mathbb Z^{k+1}}\Big(\widehat {f^+}(n_1,\dots,n_{k})+\widehat {f^-}(n_1,\dots,n_{k})\Big)+\text{error term},$$
where
\begin{equation*}
\begin{split}
\widehat {f^\pm_r}(n_1,\dots,n_{k}, n)=(-1)^{\sum_{i=1}^kn_i+n}\Big(\frac{r}{2\pi}\Big)^{k+1}&\int_{(-\pi,\pi)^{k-1}\times \mathcal D}\psi(x_1,\dots, x_k,y)\epsilon(x_k,y)\\
&e^{-x_ki+\frac{r}{4\pi i}\Big(V^\pm_r(x_1,\dots,x_k,y)-4\pi\sum_{i=1}^{k}n_ix_k-4\pi n y\Big)}dx_1\dots dx_kdy.
\end{split}
\end{equation*}
\end{proposition}

\begin{proof} To apply the Poisson Summation Formula, we need to make the sum in (\ref{halfinteger}) over integers instead of half-integers. To do this, we let $m_i'=m_i+\frac{1}{2}$ for $i=1,\dots, k$ and let $m'=m+\frac{1}{2}.$ Then 
$$\sum_{(m_1,\dots,m_k,m)\in(\mathbb Z+\frac{1}{2})^{k+1}}f^\pm_r(m_1,\dots,m_k,m)=\sum_{(m'_1,\dots,m'_k,m')\in\mathbb Z^{k+1}}f^\pm_r\Big(m'_1-\frac{1}{2},\dots,m'_k-\frac{1}{2},m'-\frac{1}{2}\Big).$$
Now by the Poisson Summation Formula, the right hand side equals
$$\sum_{(n_1,\dots,n_k,n)\in\mathbb Z^{k+1}}\int_{\mathbb R^{k+1}}f^\pm_r\Big(m'_1-\frac{1}{2},\dots,m'_k-\frac{1}{2},m'-\frac{1}{2}\Big)e^{\sum_{j=1}^k2\pi i n_jm'_j+2\pi inm'}dm'_1\dots dm'_kdm'.$$
Finally, using the change of variable $x_i=\frac{2\pi m_i}{r}=\frac{2\pi m'_i}{r}-\frac{\pi}{r}$ for $i=1,\dots, k$ and $y=\frac{2\pi m}{r}=\frac{2\pi m'}{r}-\frac{\pi}{r},$ we get the result.
\end{proof}

We will give a preliminary estimate of the Fourier coefficients in Section \ref{preest} simplifying them to a $2$-dimensional integral which will be further estimated in Sections \ref{leading} and \ref{other}, and estimate the error term in Proposition \ref{Poisson} as follows. 

\begin{proposition}\label{bound} For $\epsilon>0,$ we can choose a sufficiently small $\delta>0$ so that if one of 
$m+m_k$ and $m-m_k$ is not in $\big(\frac{\delta r}{2\pi} , \frac{r}{4}-\frac{\delta r}{2\pi} \big)\cup\big(\frac{r}{2}+\frac{\delta r}{2\pi} , \frac{3r}{4}-\frac{\delta r}{2\pi} \big),$
then
$$|g_r(m_1,\dots,m_{k}, m)|<O\Big(e^{\frac{r}{4\pi}\big(\frac{1}{2}\mathrm{Vol}(\mathrm S^3\setminus K_{4_1})+\epsilon\big)}\Big).$$
As a consequence, the error term in Proposition \ref{Poisson} is of order at most $O\Big(e^{\frac{r}{4\pi}\big(\frac{1}{2}\mathrm{Vol}(\mathrm S^3\setminus K_{4_1})+\epsilon\big)}\Big).$

\end{proposition}

To prove Proposition \ref{bound}, we need the following estimate, which first appeared in \cite[Proposition 8.2]{GL} for $t=e^{\frac{2\pi i}{r}},$ and for the root $t=e^{\frac{4\pi i}{r}}$ in \cite[Proposition 4.1]{DK}.

\begin{lemma}\label{est}
 For any integer $0<n<r$ and at $t=e^{\frac{4\pi i}{r}},$
 $$ \log\left|\{n\}!\right|=-\frac{r}{2\pi}\Lambda\left(\frac{2n\pi}{r}\right)+O\left(\log(r)\right).$$
\end{lemma}

\begin{proof}[Proof of Proposition \ref{bound}] We have
$$|g^\pm_r(m_1,\dots,m_{k}, m)|=\Big|\sin\Big(\frac{2\pi m_1}{r}\Big)\epsilon\Big(\frac{2\pi m_k}{r},\frac{2\pi m}{r}\Big)\Big|\Big|\frac{\{r-1-m-m_k\}!}{\{m-m_k\}!}\Big|,$$
and by Lemma \ref{est}, we have
$$\log |g^\pm_r(m_1,\dots,m_{k}, m)|=-\frac{r}{2\pi}\Lambda\Big(\frac{2\pi(r-1-m-m_k)}{r}\Big)+\frac{r}{2\pi}\Lambda\Big(\frac{2\pi(m-m_k)}{r}\Big)+O(\log(r)).$$
Choose $\delta>0$ so that 
$$\Lambda(\delta)<\frac{\epsilon}{4}.$$
Now if one of 
$m+m_k$ and $m-m_k$ is not in $\big(\frac{\delta r}{2\pi} , \frac{r}{4}-\frac{\delta r}{2\pi} \big)\cup\big(\frac{r}{2}+\frac{\delta r}{2\pi} , \frac{3r}{4}-\frac{\delta r}{2\pi} \big),$
then
$$\log |g^\pm_r(m_1,\dots,m_{k}, m)|<\frac{r}{2\pi}\Big(\Lambda\Big(\frac{\pi}{6}\Big)+\frac{\epsilon}{2}\Big)=\frac{r}{4\pi}\Big(\frac{1}{2}\mathrm{Vol}(\mathrm S^3\setminus K_{4_1})+\epsilon\Big).$$
The last equality is true because by properties of the Lobachevsky function $\Lambda\big(\frac{\pi}{6}\big)=\frac{3}{2}\Lambda\big(\frac{\pi}{3}\big),$ and the volume of $\mathrm S^3\setminus K_{4_1}$ equals $6\Lambda(\frac{\pi}{3}).$
\end{proof}


\section{A simplification of the Fourier coefficients}\label{preest}

The main results of this Section are Propositions \ref{0} and \ref{other'} below, which simplify the Fourier coefficients $\hat f_r(n_1,\dots,n_{k}, n).$

We consider the continued fraction expansion $$\frac{p}{q}=a_k-\frac{1}{a_{k-1}-\frac{1}{\cdots-\frac{1}{a_1}}}$$ 
of $\frac{p}{q}$ with the requirement that $a_i\geqslant 2$ for $i=1,\dots a_{k-1}.$ For each $i\leqslant k-1,$ we let
$$b_i= a_i-\frac{1}{a_{i-1}-\frac{1}{\cdots-\frac{1}{a_1}}},$$
and let
$$c_i=\prod_{j=1}^ib_i.$$
Then $c_1=b_1=a_1\geqslant 2,$ $b_k=\frac{p}{q},$  $c_{k-1}=q,$ 
and  $b_i>1$ and $c_i>c_{i-1}\geqslant 2$ for any $i\leqslant k-1.$ We need the following Lemmas.

\begin{lemma}\cite[Lemma 5.1]{WY}\label{arithmetic} Let $(p',q')$ be the a unique pair such that $pp'+qq'=1$ and $-q< p' \leqslant 0.$ Then
$$\sum_{j=1}^{k-1}\frac{1}{c_{j-1}c_j}=-\frac{p'}{q}.$$
\end{lemma}

\begin{lemma}\cite[Lemma 5.2]{WY}\label{int} Let $\beta\in\mathbb R\setminus\{0\}.$ 
\begin{enumerate}[(1)]
\item If $\alpha \in (a,b),$  then
$$\int_a^be^{\frac{r}{4\pi i}\beta(x-\alpha)^2}dx=\frac{2\pi\sqrt{i}}{\sqrt{\beta}}\frac{1}{\sqrt{r}}\Big(1+O\Big(\frac{1}{\sqrt r}\Big)\Big).$$
\item If $\alpha\in\mathbb R\setminus  (a,b),$ then
$$\Big|\int_a^be^{\frac{r}{4\pi i}\beta(x-\alpha)^2}dx\Big|\leqslant O\Big(\frac{1}{r}\Big).$$
\end{enumerate}
\end{lemma}

By completing the squares, we have on $D$
\begin{equation*}
\begin{split}
V^+_r(x_1,\dots,x_k,y)=&-\sum_{i=1}^{k-1}b_i\Big(x_i+\frac{x_{i+1}}{b_i}+(-1)^i\frac{x_0}{c_i}\Big)^2+\sum_{i=1}^{k-1}\frac{x_0^2}{c_{i-1}c_i}-a_0x_0^2\\
&-\frac{px_k^2}{q}-\frac{(-1)^k2x_0 x_k}{q}-2\pi x_k+4x_ky-\varphi_r\Big(\pi-x_k-y-\frac{\pi}{r}\Big)+\varphi_r\Big(y-x_k+\frac{\pi}{r}\Big),\\
\end{split}
\end{equation*}
and
\begin{equation*}
\begin{split}
V^-_r(x_1,\dots,x_k,y)=&-\sum_{i=1}^{k-1}b_i\Big(x_i+\frac{x_{i+1}}{b_i}-(-1)^i\frac{x_0}{c_i}\Big)^2+\sum_{i=1}^{k-1}\frac{x_0^2}{c_{i-1}c_i}-a_0x_0^2\\
&-\frac{px_k^2}{q}+\frac{(-1)^k2x_0 x_k}{q}-2\pi x_k+4x_ky-\varphi_r\Big(\pi-x_k-y-\frac{\pi}{r}\Big)+\varphi_r\Big(y-x_k+\frac{\pi}{r}\Big),\\
\end{split}
\end{equation*}
and with the corresponding change of the variables in $\varphi_r$ on $D'$ and $D''.$

From now on, we will let $x=x_k.$ Then solving the system of the critical equations
$$x_i+\frac{x_{i+1}}{b_i}\pm(-1)^i\frac{x_0}{c_i}=0$$
for $x_i$'s in terms of $x,$ we have for every $i$ in $\{1,\dots, k\}$
$$x^\pm_i(x)=(-1)^{k-i}c_{i-1}\Big(\frac{x}{q}\pm(-1)^k\sum_{j=i}^{k-1}\frac{x_0}{c_{j-1}c_j}\Big),$$
and in particular, by Lemma \ref{arithmetic},
$$x^\pm_1(x)=\frac{(-1)^{k-1}x\pm p'x_0}{q}.$$

Now we start to simplify the Fourier coefficients. We need the following

\begin{lemma}\label{in} If $-\pi+\frac{q\pi}{r}<x<\pi-\frac{q\pi}{r},$ then $-\pi+\frac{2\pi}{r}<x^\pm_i(x)<\pi-\frac{2\pi}{r}$ for each $i\in\{1,\dots,k-1\}.$
\end{lemma}
\begin{proof} We use a backward induction to prove a stronger statement that $$-\pi+\frac{c_{i-1}\pi}{r}<x^\pm_i(x)<\pi-\frac{c_{i-1}\pi}{r}$$ for each $i.$ We first note that $|x_0|\leqslant \pi.$ 

For $x^\pm_{k-1}(x),$ we have
\begin{equation*}
\begin{split}
|x^\pm_{k-1}(x)|&=\Big|\frac{x}{b_{k-1}}\pm\frac{x_0}{c_{k-1}}\Big|=\Big|\frac{c_{k-2}x}{c_{k-1}}\pm\frac{x_0}{c_{k-1}}\Big|\\
&\leqslant \frac{c_{k-2}\pi+|x_0|-\frac{c_{k-2}q\pi}{r}}{c_{k-1}}\leqslant \frac{(c_{k-2}+1)\pi-\frac{c_{k-2}q\pi}{r}}{c_{k-1}}\leqslant \pi-\frac{c_{k-2}\pi}{r},
\end{split}
\end{equation*}
where the last inequality comes from that $q=c_{k-1}>c_{k-2},$ hence $q\geqslant c_{k-2}+1.$

Now assume the result holds for $x^\pm_{i+1}(x)$ that $-\pi+\frac{c_{i}\pi}{r}<x^\pm_{i+1}(x)<\pi-\frac{c_{i}\pi}{r},$ then 
\begin{equation*}
\begin{split}|x^\pm_i(x)|&=\Big|\frac{x^\pm_{i+1}}{b_{i}}\pm\frac{x_0}{c_{i}}\Big|=\Big|\frac{c_{i-1}x^\pm_{i+1}}{c_i}\pm\frac{x_0}{c_i}\Big|\\
&\leqslant \frac{c_{i-1}\pi+|x_0|-\frac{c_{i-1}c_i\pi}{r}}{c_{i}}\leqslant \frac{(c_{i-1}+1)\pi-\frac{c_{i-1}c_i\pi}{r}}{c_{i}}\leqslant \pi-\frac{c_{i-1}\pi}{r},
\end{split}
\end{equation*}
where the last inequality comes from that $c_{i}>c_{i-1},$ hence $c_i\geqslant c_{i-1}+1.$
\end{proof}

\begin{proposition}\label{0}
\begin{equation*}
\hat f^\pm_r(0,0,\dots, 0)=\frac{i^{-\frac{k-1}{2}}r^{\frac{k+3}{2}}}{4\pi^2\sqrt{q}}\Big(\int_{\mathcal D_{\frac{\delta}{2}}}\epsilon(x,y)e^{-xi+\frac{r}{4\pi i}V^\pm_r(x,y)}dxdy\Big)\Big(1+O\Big(\frac{1}{\sqrt r}\Big)\Big),
\end{equation*}
 where
 $$V_r^\pm(x,y)=\frac{-px^2\pm(-1)^k2x_0 x}{q}-2\pi x+4xy -\varphi_r\Big(\pi-y-x_k-\frac{\pi}{r}\Big)+\varphi_r\Big(y-x_k+\frac{\pi}{r}\Big)-\Big(\frac{p'}{q}+a_0\Big)x_0^2,$$
on $D,$ and with the corresponding change of the variables in $\varphi_r$ on $D'$ and $D''.$
\end{proposition}

\begin{proof} For $(x,y)\in \mathcal D_{\frac{\delta}{2}},$ since $-\pi+\frac{q\pi}{r}<x<\pi-\frac{q\pi}{r},$ iteratively using Lemma \ref{int} (1) and Lemma \ref{in} to the variables $x_1,\dots, x_{k-1},$ we get the estimate. On the region $\big((-\pi,\pi)^{k-1}\setminus [-\pi+\frac{2\pi}{r},\pi-\frac{2\pi}{r}]^{k-1}\big)\times \mathcal D_{\frac{\delta}{2}}$ where the bump function makes a difference, by Lemma \ref{in}, for at least one variable $x_i$ Lemma \ref{int} (2) applies. Hence the contribution there is of order at most $O\big(\frac{1}{\sqrt{r}})$ times the whole integral.
\end{proof}

In Section \ref{Asy}, we will show that $\widehat{f^\pm}(0,0,\dots, 0)$ are the only leading Fourier coefficients, ie., have the largest growth rate.
\\

The other Fourier coefficients can be simplified similarly. By a completion of the squares, we have
\begin{equation*}
\begin{split}
&V^\pm_r(x_1,\dots,x_{k-1},x,y)-4\pi\sum_{i=1}^{k-1}n_ix_i-4\pi k_1x-4\pi k_2y\\
=&-\sum_{i=1}^{k-1}b_i\Big(x_i+\frac{x_{i+1}}{b_i}+\sum_{j=1}^i(-1)^{i-j}\frac{2n_jc_{j-1}\pi}{c_i}\pm(-1)^i\frac{x_0}{c_i}\Big)^2+C\\
&-\frac{px^2}{q}\mp\frac{(-1)^k2x_0 x}{q}-2\pi x+4xy-\varphi_r\Big(\pi-x-y-\frac{\pi}{r}\Big)+\varphi_r\Big(y-x+\frac{\pi}{r}\Big)\\
&-\frac{4\pi k_0 x}{q}-4\pi k_1x-4\pi k_2y,\\
\end{split}
\end{equation*}
on $D,$ where $C$ is a real constant depending on $(n_1,\dots, n_{k-1}),$ and 
$$k_0=\sum_{j=1}^{k-1}(-1)^{k-j}n_jc_{j-1}.$$
On $D'$ and $D''$ there is a corresponding change of the variables in $\varphi_r.$

Solving the system of the critical equations
\begin{equation*}
x_i+\frac{x_{i+1}}{b_i}+\sum_{j=1}^i(-1)^{i-j}\frac{2n_jc_{j-1}\pi}{c_i}\pm(-1)^i\frac{x_0}{c_i}=0,
\end{equation*}
we can write each $x_i$ as a function $x^\pm_i(x)$ of $x.$

Iteratively using Lemma \ref{int}, we have
\begin{proposition}\label{n} Let
$$V_r^{\pm,(k_0, k_1, k_2)}(x,y)=V_r^\pm(x,y)-\frac{4k_0\pi x}{q}-4k_1\pi x-4k_2 \pi y+C,$$
where $C$ is a real constant depending on $(n_1,\dots, n_{k-1}).$ Then
\begin{equation*}
\begin{split}
|\widehat {f^\pm_r}(n_1,\dots,n_{k-1},k_1,k_2)|\leqslant O\Big(r^{\frac{k+3}{2}}\Big)\Big|\int_{\mathcal D_{\frac{\delta}{2}}}e^{\frac{r}{4\pi i}V_r^{\pm,(k_0, k_1, k_2)}(x,y)}dxdy\Big|.
\end{split}
\end{equation*}
Moreover, if for each $x\in (-\pi,\pi),$ there is some $i\in\{1,\dots, k-1\}$ such that $x^\pm_i(x)\notin(-\pi, \pi),$ then
\begin{equation*}
\begin{split}|\widehat {f^\pm_r}(n_1,,\dots,n_{k-1}, k_1,k_2)|\leqslant O\Big(r^{\frac{k+2}{2}}\Big)\Big|\int_{\mathcal D_{\frac{\delta}{2}}}e^{\frac{r}{4\pi i}V_r^{\pm,(k_0, k_1, k_2)}(x,y)}dxdy\Big|.
\end{split}
\end{equation*}
\end{proposition}

 The following lemma will be need later in estimating the growth rate of the invariants.
 
 \begin{lemma}\label{4.9}

\begin{enumerate}[(1)]
\item If $(n_1,\dots,n_{k-1})\neq (0,\dots,0)$ and $k_0=0,$ then for each $x\in (-\pi,\pi)$
there is some $i\in\{1,\dots, k-1\}$ so that $x^+_i(x)\notin(-\pi, \pi),$ and some $j\in\{1,\dots, k-1\}$ so that $x^-_j(x)\notin(-\pi, \pi).$

\item  If $(n_1,\dots,n_{k-1})\neq (0,\dots,0)$ and $|k_0|\geqslant  q,$ then
$x^\pm_{k-1}(x)\notin(-\pi, \pi)$  for all $x\in (-\pi,\pi).$
\end{enumerate}
\end{lemma}

\begin{proof} We prove (1) and (2) for $x^+_i(x),$ and that for $x^-_j(x)$ is similar. 

For (1), let $x\in(-\pi,\pi)$ and let $n_{i_0}$ be the last non-zero number in $n_i$'s, ie. $n_{i_0}\neq 0$ and $n_j=0$ for all $j>i_0.$ Then $i_0\geqslant 2$ since otherwise $n_{i_0}=n_1=(-1)^{k-1}k_0=0,$ which is a contradiction.

 If $x^+_{i_0}(x)\notin(-\pi,\pi),$ then we are done. 

If $x^+_{i_0}(x)\in(-\pi,\pi),$ then we have
$$\Big|\frac{x^+_{i_0}(x)}{b_{i_0-1}}\pm\frac{x_0}{c_{i_0-1}}\Big|<\frac{\pi}{b_{i_0-1}}+\frac{\pi}{c_{i_0-1}}=\frac{c_{i_0-2}+1}{c_{i_0-1}}\pi\leqslant \pi,$$
where the equality and the last inequality come from that $c_i=b_ic_{i-1}$ and $c_{i-1}+1\leqslant c_i.$ 

Since
$$|k_0|=\Big|\sum_{j=1}^{i_0}(-1)^jn_jc_{j-1}\Big|=|n_{i_0}c_{i_0-1}|-\Big|\sum_{j=1}^{i_0-1}(-1)^jn_jc_{j-1}\Big|=0,$$
we have
$$\Big|\sum_{j=1}^{i_0-1}(-1)^j\frac{n_jc_{j-1}}{c_{i_0-1}}\Big|=|n_{i_0}|.$$
Then
\begin{equation*}
\begin{split}
|x^+_{i_0-1}|&=\Big|\frac{x^+_{i_0}(x)}{b_{i_0-1}}+\sum_{j=1}^{i_0-1}(-1)^{i_0-1-j}\frac{2n_jc_{j-1}\pi}{c_{i_0-1}}-(-1)^{i_0-1}\frac{x_0}{c_{i_0-1}}\Big|\\
&\geqslant \bigg|\Big|\sum_{j=1}^{i_0-1}(-1)^{j}\frac{2n_jc_{j-1}\pi}{c_{i_0-1}}\Big|-\Big|\frac{x^+_{i_0}(x)}{b_{i_0-1}}-(-1)^{i_0-1}\frac{x_0}{c_{i_0-1}}\Big|\bigg|\\
&\geqslant 2|n_{i_0}|\pi -\pi \geqslant \pi.
\end{split}
\end{equation*}

The proof for (2) is similar.  Since $-\pi<x<\pi,$ we have 
$$\Big|\frac{x}{b_{k-1}}\pm\frac{x_0}{q}\Big|<\frac{\pi}{b_{k-1}}+\frac{\pi}{q}=\frac{c_{k-2}+1}{q}\pi\leqslant \pi.$$
If $|k_0|\geqslant q,$ then 
\begin{equation*}
\begin{split}
|x^\pm_{k-1}|&=\Big|\frac{x}{b_{k-1}}+\frac{2k_0\pi}{q}\pm(-1)^{k-1}\frac{x_0}{q}\Big|\\
&\geqslant \bigg|\Big|\frac{2k_0\pi}{q}\Big|-\Big|\frac{x}{b_{k-1}}\pm(-1)^{k-1}\frac{x_0}{q}\Big|\bigg|\geqslant 2\pi -\pi=\pi.
\end{split}
\end{equation*}
\end{proof}

Let
$$\widehat {F^\pm_r}(k_0, k_1, k_2)=\int_{\mathcal D_{\frac{\delta}{2}}}e^{\frac{r}{4\pi i}V_r^{\pm,(k_0, k_1, k_2)}(x,y)}dxdy.$$
Then by Propositions \ref{n} and Lemma \ref{4.9}, we have the following

\begin{proposition}\label{other'}
\begin{enumerate}[(1)]
\item For $(n_1,\dots,n_k,n)\neq (0,\dots,0),$  
\begin{equation*}
\begin{split}
\Big|\widehat {f^\pm_r}(n_1,\dots, n_k,n)\Big|\leqslant O\Big(r^{\frac{k+3}{2}}\Big)\Big|\widehat{F^\pm_r}(k_0,k_1,k_2)\Big|.
\end{split}
\end{equation*}

\item  If $(n_1,\dots,n_k,n)\neq (0,\dots,0)$ with $k_0=0$ or with $|k_0|\geqslant q,$ then
\begin{equation*}
\begin{split}
\Big|\widehat {f^\pm_r}(n_1,\dots, n_k,n)\Big|\leqslant O\Big(r^{\frac{k+2}{2}}\Big)\Big|\widehat{F^\pm_r}(k_0,k_1,k_2)\Big|.
\end{split}
\end{equation*}
\end{enumerate}
\end{proposition}
 The integrals $\hat F_r(k_0,k_1,k_2)$ will be further estimated in Section \ref{Asy}.
 \\

 We notice that $V_r^\pm(x,y)$ and $V_r^{\pm,(k_0, k_1, k_2)}(x,y)$ define holomorphic functions on the following regions $D_{\mathbb C, \delta},$ $D'_{\mathbb C, \delta}$ and $D''_{\mathbb C,\delta}$ of $\mathbb C^2,$  where for $\delta\geqslant 0,$
$$D_{\mathbb C, \delta}=\Big\{(x,y)\in \mathbb C^2\ \Big|\ \delta<\mathrm{Re}(y)+\mathrm{Re}(x)<\frac{\pi}{2}-\delta, \delta < \mathrm{Re}(y)-\mathrm{Re}(x)< \frac{\pi}{2}-\delta\Big\},$$
$$D'_{\mathbb C,\delta}=\Big\{(x,y)\in \mathbb C^2\ \Big|\ \delta < \mathrm{Re}(y)+\mathrm{Re}(x)<\frac{\pi}{2}-\delta, \pi+\delta< \mathrm{Re}(y)-\mathrm{Re}(x) < \frac{3\pi}{2}-\delta \Big\}$$
and
$$D''_{\mathbb C,\delta}=\Big\{(x,y)\in \mathbb C^2\ \Big|\ \pi+\delta < \mathrm{Re}(y)+\mathrm{Re}(x) < \frac{3\pi}{2}-\delta, \delta<\mathrm{Re}(y)-\mathrm{Re}(x) <\frac{\pi}{2}-\delta\Big\}.$$

When $\delta=0,$ we denote the corresponding regions by $D_{\mathbb C},$ $D'_{\mathbb C}$ and $D''_{\mathbb C}.$
 
We consider the following holomorphic functions
 $$V^\pm(x,y)=\frac{-px^2\pm2x_0 x}{q}-2\pi x+4xy-\mathrm{Li}_2(e^{-2i(y+x)})+\mathrm{Li}_2(e^{2i(y-x)})-\Big(\frac{p'}{q}+a_0\Big)x_0^2$$
 on $D_{\mathbb C},$  $D'_{\mathbb C}$ and $D''_{\mathbb C},$ which will play a crucial role in Section \ref{Asy} in the estimate of the Fourier coefficients.


\section{Geometry of the critical points}\label{geometry}

The main result of this Section is Proposition \ref{Vol} which shows that the critical value of the functions $V^\pm$ defined in the previous section has real part the volume of $M_{K_\theta}$ and imaginary part the Chern-Simons invariant of $M_{K_\theta}$ as defined in Section \ref{CCS}. The key observation is Lemma \ref{=} that the system of critical point equations of $V^\pm$ is equivalent to the system of hyperbolic gluing equations (consisting of an edge equation and an equation of the $\frac{p}{q}$ Dehn-filling with prescribed cone angle) for a particular ideal triangulation of the figure-$8$ knot complement.

According to Thurston's notes \cite{T}, the complement of the figure-$8$ knot has an ideal triangulation as drawn in Figure \ref{figure-8}. We let $A$ and $B$ be the shape parameters of the two ideal tetrahedra 
and let $A'=\frac{1}{1-A},$ $A''=1-\frac{1}{A},$ $B'=\frac{1}{1-B}$ and $B''=1-\frac{1}{B}.$ 

\begin{figure}[htbp]
\centering
\includegraphics[scale=0.3]{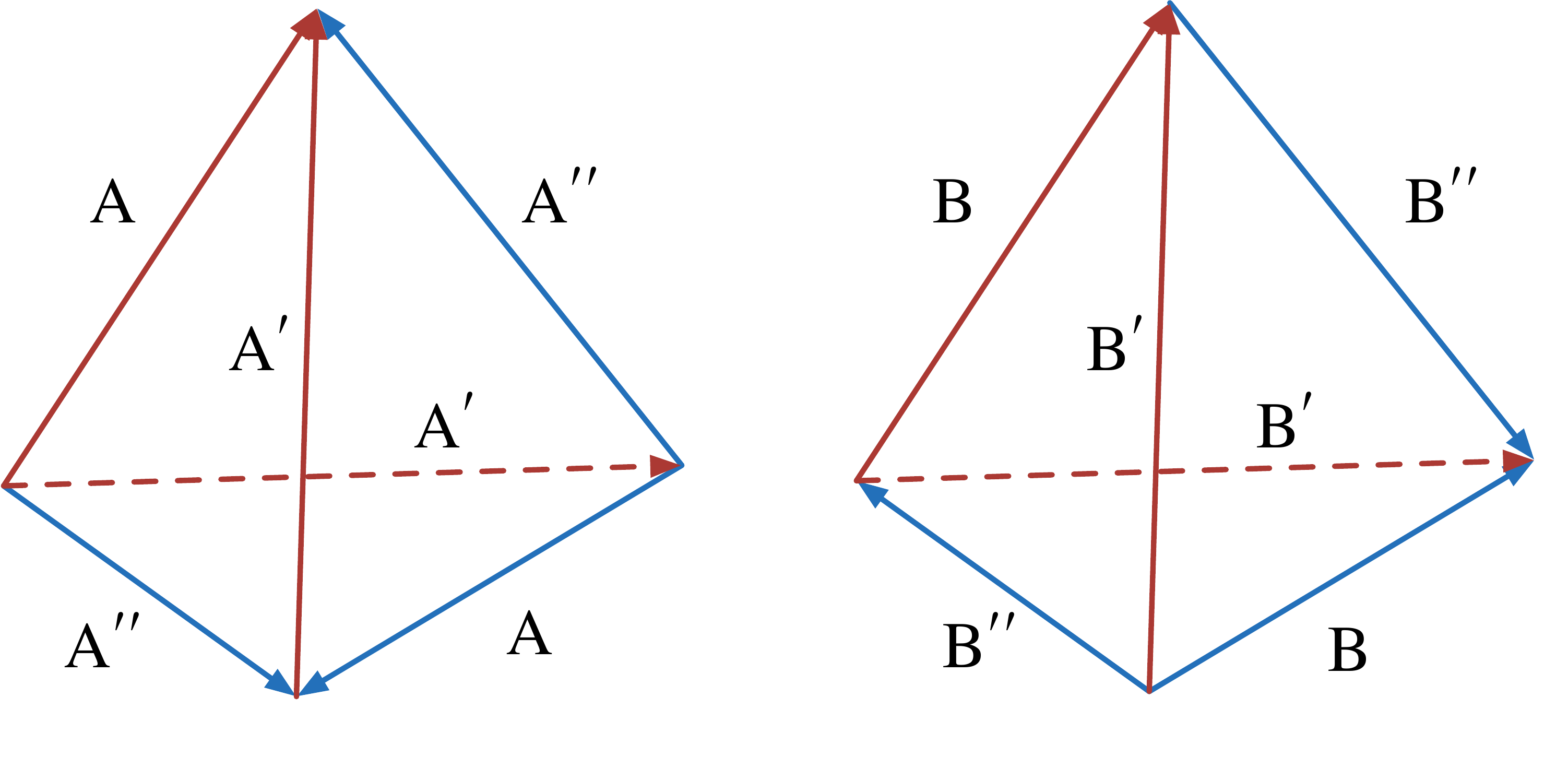}
\caption{An ideal triangulation of the figure-$8$ knot complement}
\label{figure-8}
\end{figure}

In Figure \ref{holonomy} is a fundamental domain of the boundary of a tubular neighborhood of $K_{4_1}.$

\begin{figure}[htbp]
\centering
\includegraphics[scale=0.3]{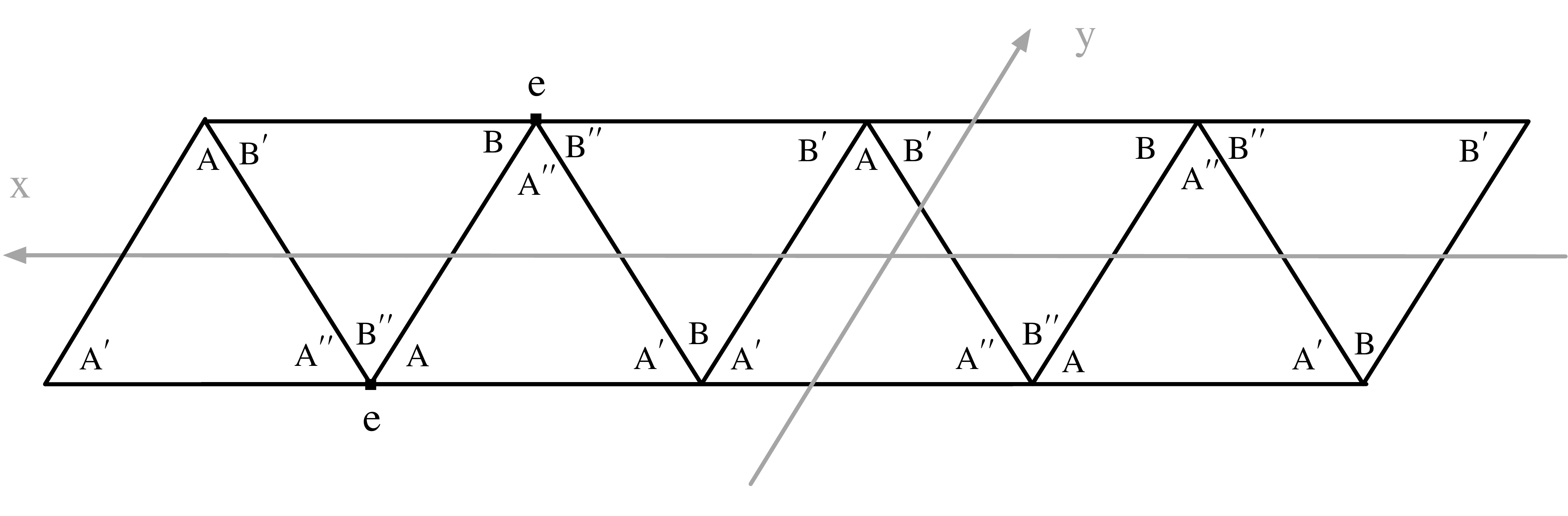}
\caption{Combinatorics around the boundary}
\label{holonomy}
\end{figure}

Recall that for $z\in \mathbb C\setminus (-\infty,0],$ the logarithmic function is defined by
$$\log z=\ln|z|+i\arg z$$ with $-\pi<\arg z<\pi.$

Then the holonomy around the edge $e$ is
$$\mathrm{H}(e)=\log A+2\log A''+\log B+2\log B'',$$
and the holonomies of the curves $x$ and $y$ are respectively
$$\mathrm{H}(x)=2\log B+2\log B''-2\log A-2\log A''$$ and
$$\mathrm{H}(y)=\log B'-\log A''.$$

By \cite{T}, we can choose the meridian $m=y$ and the longitude $l=x+2y.$ Hence
$$\mathrm{H}(m)=\log B'-\log A'',$$
and
\begin{equation*}
\begin{split}
\mathrm{H}(l)=2\pi i-2\log A-4\log A''.
\end{split}
\end{equation*}
Then for the incomplete hyperbolic metric that gives the hyperbolic cone metric with cone angle $\theta,$  the system of hyperbolic gluing equations 
\begin{equation*}
\left \{\begin{array}{l}\mathrm{H}(e)=2\pi i\\
\\
p \mathrm{H}(m)+q\mathrm{H}(l)= \theta i
\end{array}\right.
\end{equation*} 
can be written as
\begin{equation}\label{HG}
\left \{\begin{array}{l}
\log A+2\log A''+\log B+2\log B''=2\pi i\\
\\
p(\log B'-\log A'')+q(2\pi i-2\log A-4\log A'')=\theta i.
\end{array}\right.
\end{equation}

From now on, we let $\theta=|2x_0|$ and, by switching the $+$ and $-$ if necessary, let 
$$V^\pm(x,y)=\frac{-px^2\pm \theta x}{q}-2\pi x+4xy-\mathrm{Li}_2(e^{-2i(y+x)})+\mathrm{Li}_2(e^{2i(y-x)})-\Big(\frac{p'}{q}+a_0\Big)\frac{\theta^2}{4}.$$
 
 Taking partial derivatives of $V^\pm,$ we have 
$$\frac{\partial V^\pm}{\partial x}=\frac{-2px\pm \theta}{q}+4y-2\pi-2i\log(1-e^{-2i(y+x)})+2i\log(1-e^{2i(y-x)})$$
and
$$\frac{\partial V^\pm}{\partial y}=4x-2i\log(1-e^{-2i(y+x)})-2i\log(1-e^{2i(y-x)}).$$
Hence the system of critical point equations of $V^\pm(x,y)$ is 
\begin{equation}\label{C}
\left \{\begin{array}{l}
4x-2i\log(1-e^{-2i(y+x)})-2i\log(1-e^{2i(y-x)})=0\\
\\
\frac{-2px\pm u}{q}+4y-2\pi-2i\log(1-e^{-2i(y+x)})+2i\log(1-e^{2i(y-x)})=0.
\end{array}\right.
\end{equation}

\begin{lemma}\label{=} 
\begin{enumerate}[(1)]
\item In $D_{\mathbb C},$ if we let $A=e^{2i(y+x)}$ and $B=e^{2i(y-x)},$  then
 the system of critical point equations (\ref{C}) of $V^+$ is equivalent to the system of hyperbolic glueing equations (\ref{HG}).
 
 \item In $D_{\mathbb C},$ if we let $A=e^{2i(y-x)}$ and $B=e^{2i(y+x)},$ then the system of critical point equations (\ref{C}) of $V^-$ is equivalent to the system of hyperbolic glueing equations (\ref{HG}).
 \end{enumerate}
\end{lemma}
 
\begin{proof} 
For (1), in $D_{\mathbb C}$ we have
\begin{equation*}
\left \{\begin{array}{l}
\log A= 2i(y+x),\\
 \log A'=\pi i-2i(y+x)-\log(1-e^{-2i(y+x)}),\\
 \log A''=\log(1-e^{-2i(y+x)}),\\
 \log B=2i(y-x),\\
\log B'=-\log(1-e^{2i(y-x)}),\\
 \log B''=\pi i-2i(y-x)+\log (1-e^{2i(y-x)}).
\end{array}\right.
\end{equation*}

For one direction, we assume that $(x,y)\in D_{\mathbb C}$ is a solution of the system of critical equations (\ref{C}) with the ``$+$'' chosen. Then from the first equation of (\ref{C})
\begin{equation*}
\begin{split}
\mathrm{H}(e)=\log A+2\log A''+\log B+2\log B''=2\pi i,
\end{split}
\end{equation*}
hence the edge equation is satisfied. Next, we compute $\mathrm{H}(m)$ and $\mathrm{H}(l).$  From the first equation of (\ref{C}), we have
\begin{equation}\label{m}
\begin{split}
\mathrm{H}(m)=\log B'-\log A''=2x i;
\end{split}
\end{equation}
and from (\ref{m}) we have
\begin{equation}\label{l}
\begin{split}
\mathrm{H}(l)&=2\pi i-2\log A-4\log A''\\
&=2\pi i-2\log A+(4xi-2\log B')-2\log A''\\
&=-4yi+2\pi i-2\log(1-e^{-2i(y+x)})+2\log(1-e^{2i(y-x)}).
\end{split}
\end{equation}
Equations (\ref{m}), (\ref{l}) and the second equation of (\ref{C}) then imply that 
$$\frac{p\mathrm{H}(m)i+\theta}{q}+\mathrm{H}(l)i=0,$$
which is equivalent to the $\frac{p}{q}$ Dehn-filling equation with cone angle $\theta$
$$p\mathrm{H}(m)+q\mathrm{H}(l)=\theta i.$$

For the other direction, assume that $(A,B)$ is a  solution of (\ref{HG}). Then the edge equation implies the first equation of (\ref{C}); and (\ref{m}), (\ref{l}) and the Dehn-filling equation with cone angle $\theta$ imply that the second equation of (\ref{C}).
\medskip

For (2), we have \begin{equation*}
\left \{\begin{array}{l}
 \log A=2i(y-x),\\
\log A'=-\log(1-e^{2i(y-x)}),\\
 \log A''=\pi i-2i(y-x)+\log (1-e^{2i(y-x)}),\\
\log B= 2i(y+x),\\
 \log B'=\pi i-2i(y+x)-\log(1-e^{-2i(y+x)}),\\
 \log B''=\log(1-e^{-2i(y+x)}),\\
\end{array}\right.
\end{equation*} 
in $D_{\mathbb C},$ and the rest of the proof is very similar to that of (1). Namely, by a computation, we have
$$\mathrm{H}(e)=2\pi i$$
which gives the edge equation. We also have
\begin{equation}\label{m'}
\mathrm{H}(m)=-2xi,
\end{equation}
and
\begin{equation}\label{l'}
\mathrm{H}(l)=4yi-2\pi i+2\log(1-e^{-2i(y+x)})-2\log(1-e^{2i(y-x)}),
\end{equation}
which, together with the second equation of (\ref{C}) imply that
$$\frac{-p\mathrm{H}(m)i-\theta}{q}-\mathrm{H}(l)i=0,$$
which is equivalent to the $\frac{p}{q}$ Dehn-filling equation with cone angle $\theta$
$$p\mathrm{H}(m)+q\mathrm{H}(l)=\theta i.$$

For the other direction, assume that $(A,B)$ is a  solution of (\ref{HG}). Then the edge equation implies the first equation of (\ref{C}); and (\ref{m'}), (\ref{l'}) and the Dehn-filling equation with cone angle $\theta$ imply that the second equation of (\ref{C}).
\end{proof}

By Thurston's notes\,\cite{T} and Hodgson\,\cite{H} (see also \cite[Section 5.7]{CHK}), for each relatively primed $(p,q)\neq (\pm 1,0)$ and every $u\in (0,2\pi),$ there is a unique solution $A_0$ and $B_0$ of (\ref{HG}) with $\mathrm{Im}A_0>0$ and $\mathrm{Im}B_0>0.$ Then by Lemma \ref{=}, we have 

\begin{corollary}\label{c} For $(p,q)\neq (\pm 1,0),$ the point $$(x_0,y_0)=\Big(\frac{\log A_0-\log B_0}{4i},\frac{\log A_0+\log B_0}{4i}\Big)$$ is the unique critical point of $V^+$ in $D_{\mathbb C},$ and $(-x_0,y_0)$ is the unique critical point of $V^-$ in $D_{\mathbb C}.$
\end{corollary}

\begin{proposition}\label{Vol} We have
\begin{enumerate}[(1)]
\item
 $$V^+(x_0,y_0)=V^-(-x_0,y_0)=i\Big(\mathrm{Vol}(M_{K_\theta})+i\mathrm{CS}(M_{K_\theta})\Big)\quad\mathrm{mod}\ \pi^2\mathbb Z.$$
\item
$$\det(\mathrm{Hess}V^+)(x_0,y_0)=\det(\mathrm{Hess}V^-)(-x_0,y_0)\neq 0.$$
\end{enumerate}
 \end{proposition}

\begin{proof}For (1), we have for $(x,y)\in D_{\mathbb C}$ that
 \begin{equation*}
 \begin{split}
-\mathrm{Li}_2(e^{-2i(y\pm x)})&=\mathrm{Li}_2(e^{2i(y\pm x)})+\frac{\pi^2}{6}+\frac{1}{2}\big(\log(-e^{2i(y\pm x)})\big)^2\\
&=\mathrm{Li}_2(e^{2i(y\pm x)})+\frac{\pi^2}{6}-2y^2-2x^2-\pi^2\mp 4xy+2\pi y\pm 2\pi x,\\
 \end{split}
 \end{equation*}
where the first equality comes from (\ref{Li2}), and the second equality comes from that $0<Re(y)\pm Re(x)<\frac{\pi}{2},$ and hence
$$\log(-e^{2i(y\pm x)})=2i(y\pm x)-\pi i.$$ 
From this, we have \begin{equation}\label{equal}
 \begin{split}
V^+(x,y)&=V^-(-x,y)\\
 &=\Big(-\frac{p}{q}-2\Big)x^2+\frac{\theta x}{q}-2y^2+2\pi y-\frac{5\pi^2}{6}+\mathrm{Li}_2(e^{2i(y+x)})+\mathrm{Li}_2(e^{2i(y-x)}).
 \end{split}
 \end{equation}
 In particular, 
 $$V^+(x_0,y_0)=V^-(-x_0,y_0).$$

Now since $V^\pm(\pm x_0,y_0)$ are the same, it suffices to show that $V^+(x_0,y_0)=i(\mathrm{Vol}(M_{K_\theta})+i\mathrm{CS}(M_{K_\theta}))$ $\mathrm{mod}$ $\pi^2\mathbb Z.$ To this end, we follow the discussion in Section \ref{CCS}. Suppose $M_{K_\theta}$ is obtained from the complement of a hyperbolic knot $K$ in $M$ by attaching a solid torus $T$ with cone angle $\theta$ along the core to the boundary $T$ of $M.$ Let $m$ and $l$ be two generators of $\pi_1(T),$ let $pm+ql$ be the curve that bounds a disk in the attached solid torus and let $\gamma'=-q'm+p'l$ where $(p',q')$ is the unique pair such that $pp'+qq'=1$ and $-q<p'\leqslant 0.$ Then the chosen framing $\gamma=\gamma'+a_0(pm+ql).$  If $\mathrm H(m),$ $\mathrm H(l)$ and $\mathrm H(\gamma)$ are respectively the holonomy of $m,$ $l$ and $\gamma,$ then by Definition \ref{CS},
$$\mathrm{Vol}(M_{K_\theta})+i\mathrm{CS}(M_{K_\theta})=\frac{\Phi(\mathrm{H}(m))}{i}-\frac{\mathrm H(m)\mathrm H(l)}{4i}+\frac{\theta\mathrm H(\gamma)}{4}\quad\mathrm{mod}\ i\pi^2\mathbb Z,$$
where $\Phi$ is the function (see Neumann-Zagier\,\cite{NZ}) defined on the deformation space of hyperbolic structures on $M\setminus K$ parametrized by $\mathrm H(m),$ characterized by 
\begin{equation}\label{char}
\left \{\begin{array}{l}
\frac{\partial \Phi(\mathrm H(m))}{\partial \mathrm H(m)}=\frac{\mathrm H(l)}{2},\\
\\
\Phi(0)=i\Big(\mathrm{Vol}(\mathrm M\setminus K)+i\mathrm{CS}(\mathrm M\setminus K)\Big).
\end{array}\right.
\end{equation} 
We will show that 
\begin{equation}\label{Phi}
\Phi(\mathrm H(m))=4x_0y_0-2\pi x_0-\mathrm{Li}_2(e^{-2(y_0+x_0)})+\mathrm{Li}_2(e^{2i(y_0-x_0)}),
\end{equation}
\begin{equation}\label{uv}
-\frac{\mathrm H(m)\mathrm H(l)}{4}=-\frac{px_0^2}{q}+\frac{\theta x_0}{2q},
\end{equation}
and 
\begin{equation}\label{gamma}
\frac{\theta i}{4}\mathrm H(\gamma)=\frac{\theta x_0 }{2q}-\Big(\frac{p'}{q}+a_0\Big)\frac{\theta^2}{4},
\end{equation}
from which the result follows.
\\

For (\ref{Phi}), we let 
$$U(x,y)=4xy-2\pi x-\mathrm{Li}_2(e^{-2(y+x)})+\mathrm{Li}_2(e^{2i(y-x)}),$$ and define
$$\Psi(u)=U(x,y(x)),$$
where  $u=2x i$  and $y(x)$ is such that 
$$\frac{\partial V^+}{\partial y}\Big|_{(x,y(x))}=0.$$

Since 
$$\frac{\partial U}{\partial y}=\frac{\partial V^+}{\partial y}\quad\text{and}\quad\frac{\partial V^+}{\partial y}\Big|_{(x,y(x))}=0,$$
we have
\begin{equation*}
\begin{split}
\frac{\partial \Psi(u)}{\partial u}=\Big(\frac{\partial U}{\partial x}+\frac{\partial U}{\partial y}\Big|_{(x,y(x))}\frac{\partial y}{\partial x}\Big)\frac{\partial x}{\partial u}=\frac{\partial U}{\partial x}\frac{\partial x}{\partial u}=\frac{\mathrm H(l)}{2},
\end{split}
\end{equation*}
where the last equality comes from (\ref{l}). Also, a direct computation shows $y(0)=\frac{\pi}{6},$ and hence
$$\Psi(0)=U\Big(0,\frac{\pi}{6}\Big)=4i\Lambda\Big(\frac{\pi}{6}\Big)=i\Big(\mathrm{Vol}(\mathrm S^3\setminus K_{4_1})+i\mathrm{CS}(\mathrm S^3\setminus K_{4_1})\Big)=i\Big(\mathrm{Vol}(\mathrm M\setminus K)+i\mathrm{CS}(\mathrm M\setminus K)\Big).$$  Therefore, $\Psi$ satisfies (\ref{char}), and hence $\Psi(u)=\Phi(u).$ 

Since $y(x_0)=y_0,$ and by (\ref{m}) $\mathrm H(m)=2x_0 i,$ we have
$$\Phi(\mathrm H(m))=\Psi(2x_0 i)=U(x_0,y_0),$$
which verifies (\ref{Phi}).

For (\ref{uv}), we have by (\ref{m}) that $\mathrm H(m)=2x_0i$ and 
$$\mathrm H(l)=\frac{\theta i-p\mathrm H(m)}{q}=\frac{\theta i-2px_0i}{q}.$$
Then $$\mathrm H(m)\mathrm H(l)=2x_0i\cdot\frac{\theta i-2px_0i}{q}=-4\Big(-\frac{px_0^2}{q}+\frac{\theta x_0}{2q}\Big)$$
from which (\ref{uv}) follows, and
\begin{equation*}
\begin{split}
\mathrm H(\gamma)=&-q'\mathrm H(m)+p'\mathrm H(l)+a_0\big(p\mathrm H(m)+q\mathrm H(l)\big)\\
=&-q'\cdot 2x_0i+p'\cdot\frac{\theta i-2px_0i}{q}+a_0 \theta i\\
=&\frac{4}{\theta i}\Big(\frac{\theta x_0}{2q}-\Big(\frac{p'}{q}+a_0\Big)\frac{\theta^2}{4}\Big)
\end{split}
\end{equation*}
from which (\ref{gamma}) follows.
\\
 
 For (2), we have by (\ref{equal})
$$\mathrm{Hess}V^+(x_0,y_0)=\begin{bmatrix}
-\frac{2p}{q}-4-\frac{4e^{2i(y_0+x_0)}}{1-e^{2i(y_0+x_0)}}-\frac{4e^{2i(y_0-x_0)}}{1-e^{2i(y_0-x_0)}} & -\frac{4e^{2i(y_0+x_0)}}{1-e^{2i(y_0+x_0)}}+\frac{4e^{2i(y_0-x_0)}}{1-e^{2i(y_0-x_0)}} \\
-\frac{4e^{2i(y_0+x_0)}}{1-e^{2i(y_0+x_0)}}+\frac{4e^{2i(y_0-x_0)}}{1-e^{2i(y_0-x_0)}}  & 4-\frac{4e^{2i(y_0+x_0)}}{1-e^{2i(y_0+x_0)}}-\frac{4e^{2i(y_0-x_0)}}{1-e^{2i(y_0-x_0)}} 
  \end{bmatrix}$$
 and
  $$\mathrm{Hess}V^-(-x_0,y_0)=\begin{bmatrix}
-\frac{2p}{q}-4-\frac{4e^{2i(y_0-x_0)}}{1-e^{2i(y_0-x_0)}}-\frac{4e^{2i(y_0+x_0)}}{1-e^{2i(y_0+x_0)}} & -\frac{4e^{2i(y_0-x_0)}}{1-e^{2i(y_0-x_0)}}+\frac{4e^{2i(y_0+x_0)}}{1-e^{2i(y_0+x_0)}} \\
-\frac{4e^{2i(y_0-x_0)}}{1-e^{2i(y_0-x_0)}}+\frac{4e^{2i(y_0+x_0)}}{1-e^{2i(y_0+x_0)}}  & 4-\frac{4e^{2i(y_0-x_0)}}{1-e^{2i(y_0-x_0)}}-\frac{4e^{2i(y_0+x_0)}}{1-e^{2i(y_0+x_0)}} 
  \end{bmatrix}.$$
  Hence 
$$\det(\mathrm{Hess}V^+)(x_0,y_0)=\det(\mathrm{Hess}V^-)(-x_0,y_0).$$
By Lemma \ref{convexity} below, the real part of the $\mathrm{Hess}V^\pm$ is positive definite. Then by \cite[Lemma]{L}, it is nonsingular.
\end{proof}

\begin{lemma}\label{convexity}
 In $D_{\mathbb C},$  $\mathrm{Im}V^\pm(x,y)$ is strictly concave down in $\mathrm{Re}(x)$ and $\mathrm{Re}(y),$ and is strictly concave up in  $\mathrm{Im}(x)$ and $\mathrm{Im}(y).$
\end{lemma}

\begin{proof}
Using (\ref{equal}), taking second derivatives $\mathrm{Im}V^\pm$ with respect to $\mathrm{Re}(x)$ and $\mathrm{Re}(y),$ we
\begin{equation*}
\begin{split}
\mathrm{Hess}(\mathrm{Im}V^\pm)&= \begin{bmatrix}
-\frac{4\mathrm{Im}e^{2i(y+x)}}{|1-e^{2i(y+x)}|^2}-\frac{4\mathrm{Im}e^{2i(y-x)}}{|1-e^{2i(y-x)}|^2} & -\frac{4\mathrm{Im}e^{2i(y+x)}}{|1-e^{2i(y+x)}|^2}+\frac{4\mathrm{Im}e^{2i(y-x)}}{|1-e^{2i(y-x)}|^2}\\
&\\
-\frac{4\mathrm{Im}e^{2i(y+x)}}{|1-e^{2i(y+x)}|^2}+\frac{4\mathrm{Im}e^{2i(y-x)}}{|1-e^{2i(y-x)}|^2} & -\frac{4\mathrm{Im}e^{2i(y+x)}}{|1-e^{2i(y+x)}|^2}-\frac{4\mathrm{Im}e^{2i(y-x)}}{|1-e^{2i(y-x)}|^2}
  \end{bmatrix}\\
&=-\begin{bmatrix}
2&-2\\
2 & 2 
  \end{bmatrix}
  \begin{bmatrix}
\frac{\mathrm{Im}e^{2i(y+x)}}{|1-e^{2i(y+x)}|^2}& 0\\
0 & \frac{\mathrm{Im}e^{2i(y-x)}}{|1-e^{2i(y-x)}|^2}
  \end{bmatrix}
  \begin{bmatrix}
2&2\\
-2 & 2 
  \end{bmatrix}.
  \end{split}
  \end{equation*}
Since in $D_{\mathbb C},$ $\mathrm{Im}e^{2i(y+x)}>0$ and $\mathrm{Im}e^{2i(y-x)}>0,$ the diagonal matrix in the middle is positive definite, and hence $\mathrm{Hess}(\mathrm{Im}V^\pm)$ is negative definite. Therefore, $\mathrm{Im}V$ is concave down in $\mathrm{Re}(x)$ and $\mathrm{Re}(y).$ Since $\mathrm{Im}V^\pm$ is harmonic, it is concave up in $\mathrm{Im}(x)$ and $\mathrm{Im}(y).$   
\end{proof}

The following Lemma will be needed later in the estimate of the Fourier coefficients.

\begin{lemma}\label{n0} $\mathrm{Im}(x_0)\neq 0.$
\end{lemma}

\begin{proof} By (\ref{m}), the holonomy of the meridian $\mathrm{H}(m)=2x_0i.$ We prove by contradiction. Suppose $\mathrm{Im}(x_0)=0,$ then $\mathrm{H}(m)$ is purely imaginary. As a consequence, $\mathrm{H}(l)=\frac{\theta i -p\mathrm{H}(m)}{q}$ is also purely imaginary.  This implies that the holonomy of the core curve of the filled solid torus $H(\gamma)=q'\mathrm H(m)-p'\mathrm H(l)$ is purely imaginary, ie. $\gamma$ has length zero. This is a contradiction.
\end{proof}


\section{Asymptotics}\label{Asy}

\subsection{Asymptotics of the leading Fourier coefficients}\label{leading'}

The main tool we use is Proposition \ref{saddle}, which is a generalization of the standard Saddle Point Approximation\,\cite{O}. A proof of Proposition \ref{saddle} could be found in \cite[Appendix A]{WY2}.

\begin{proposition}\label{saddle}
Let $D_{\mathbf z}$ be a region in $\mathbb C^n$ and let $D_{\mathbf a}$ be a region in $\mathbb R^k.$ Let $f(\mathbf z,\mathbf a)$ and $g(\mathbf z,\mathbf a)$ be complex valued functions on $D_{\mathbf z}\times D_{\mathbf a}$  which are holomorphic in $\mathbf z$ and smooth in $\mathbf a.$ For each positive integer $r,$ let $f_r(\mathbf z,\mathbf a)$ be a complex valued function on $D_{\mathbf z}\times D_{\mathbf a}$ holomorphic in $\mathbf z$ and smooth in $\mathbf a.$
For a fixed $\mathbf a\in D_{\mathbf a},$ let $f^{\mathbf a},$ $g^{\mathbf a}$ and $f_r^{\mathbf a}$ be the holomorphic functions  on $D_{\mathbf z}$ defined by
$f^{\mathbf a}(\mathbf z)=f(\mathbf z,\mathbf a),$ $g^{\mathbf a}(\mathbf z)=g(\mathbf z,\mathbf a)$ and $f_r^{\mathbf a}(\mathbf z)=f_r(\mathbf z,\mathbf a).$ Suppose $\{\mathbf a_r\}$ is a convergent sequence in $D_{\mathbf a}$ with $\lim_r\mathbf a_r=\mathbf a_0,$ $f_r^{\mathbf a_r}$ is of the form
$$ f_r^{\mathbf a_r}(\mathbf z) = f^{\mathbf a_r}(\mathbf z) + \frac{\upsilon_r(\mathbf z,\mathbf a_r)}{r^2},$$
$\{S_r\}$ is a sequence of embedded real $n$-dimensional closed disks in $D_{\mathbf z}$ sharing the same boundary, and $\mathbf c_r$ is a point on $S_r$ such that $\{\mathbf c_r\}$ is convergent  in $D_{\mathbf z}$ with $\lim_r\mathbf c_r=\mathbf c_0.$ If for each $r$
\begin{enumerate}[(1)]
\item $\mathbf c_r$ is a critical point of $f^{\mathbf a_r}$ in $D_{\mathbf z},$
\item $\mathrm{Re}f^{\mathbf a_r}(\mathbf c_r) > \mathrm{Re}f^{\mathbf a_r}(\mathbf z)$ for all $\mathbf z \in S\setminus \{\mathbf c_r\},$
\item the Hessian matrix $\mathrm{Hess}(f^{\mathbf a_r})$ of $f^{\mathbf a_r}$ at $\mathbf c_r$ is non-singular,
\item $|g^{\mathbf a_r}(\mathbf c_r)|$ is bounded from below by a positive constant independent of $r,$
\item $|\upsilon_r(\mathbf z, \mathbf a_r)|$ is bounded from above by a constant independent of $r$ on $D_{\mathbf z},$ and
\item  the Hessian matrix $\mathrm{Hess}(f^{\mathbf a_0})$ of $f^{\mathbf a_0}$ at $\mathbf c_0$ is non-singular,
\end{enumerate}
then
\begin{equation*}
\begin{split}
 \int_{S_r} g^{\mathbf a_r}(\mathbf z) e^{rf_r^{\mathbf a_r}(\mathbf z)} d\mathbf z= \Big(\frac{2\pi}{r}\Big)^{\frac{n}{2}}\frac{g^{\mathbf a_r}(\mathbf c_r)}{\sqrt{-\det\mathrm{Hess}(f^{\mathbf a_r})(\mathbf c_r)}} e^{rf^{\mathbf a_r}(\mathbf c_r)} \Big( 1 + O \Big( \frac{1}{r} \Big) \Big).
 \end{split}
 \end{equation*}
\end{proposition}

Let $(x_0, y_0)$ be the unique critical point of $V^+$ in $D_{\mathbb C},$ and by Corollary \ref{c} $(-x_0, y_0)$ is the unique critical point of $V^-$ in $D_{\mathbb C}.$ Let $\delta$ be as in Proposition \ref{bound}, and as drawn in Figure \ref{surface} let $S^+= S^+_{\text{top}}\cup S^+_{\text{side}}\cup(D_{\frac{\delta}{2}}\setminus D_{\delta})$ be the union of $D_{\frac{\delta}{2}}\setminus D_{\delta}$ with the two surfaces 
$$S^+_{\text{top}}=\Big\{(x,y)\in D_{\mathbb C,\delta}\ |\ (\mathrm{Im}(x),\mathrm{Im}(y))=(\mathrm{Im}(x_0),\mathrm{Im}(y_0))\Big\}$$ and 
$$S^+_{\text{side}}=\Big\{ (\theta_1+it\mathrm{Im}(x_0), \theta_2+it \mathrm{Im}(y_0)) \ |\ (\theta_1,\theta_2)\in \partial D_{\delta}, t\in[0,1]\Big\};$$
and let  $S^-= S^-_{\text{top}}\cup S^-_{\text{side}}\cup(D_{\frac{\delta}{2}}\setminus D_{\delta})$ be the union of $D_{\frac{\delta}{2}}\setminus D_{\delta}$ with the two surfaces 
$$S^-_{\text{top}}=\Big\{(x,y)\in D_{\mathbb C,\delta}\ |\ (\mathrm{Im}(x),\mathrm{Im}(y))=(-\mathrm{Im}(x_0),\mathrm{Im}(y_0))\Big\}$$ and 
$$S^-_{\text{side}}=\Big\{ (\theta_1-it\mathrm{Im}(x_0), \theta_2+it\mathrm{Im}(y_0)) \ |\ (\theta_1,\theta_2)\in \partial D_{\delta}, t\in[0,1]\Big\}.$$

\begin{figure}[htbp]
\centering
\includegraphics[scale=0.3]{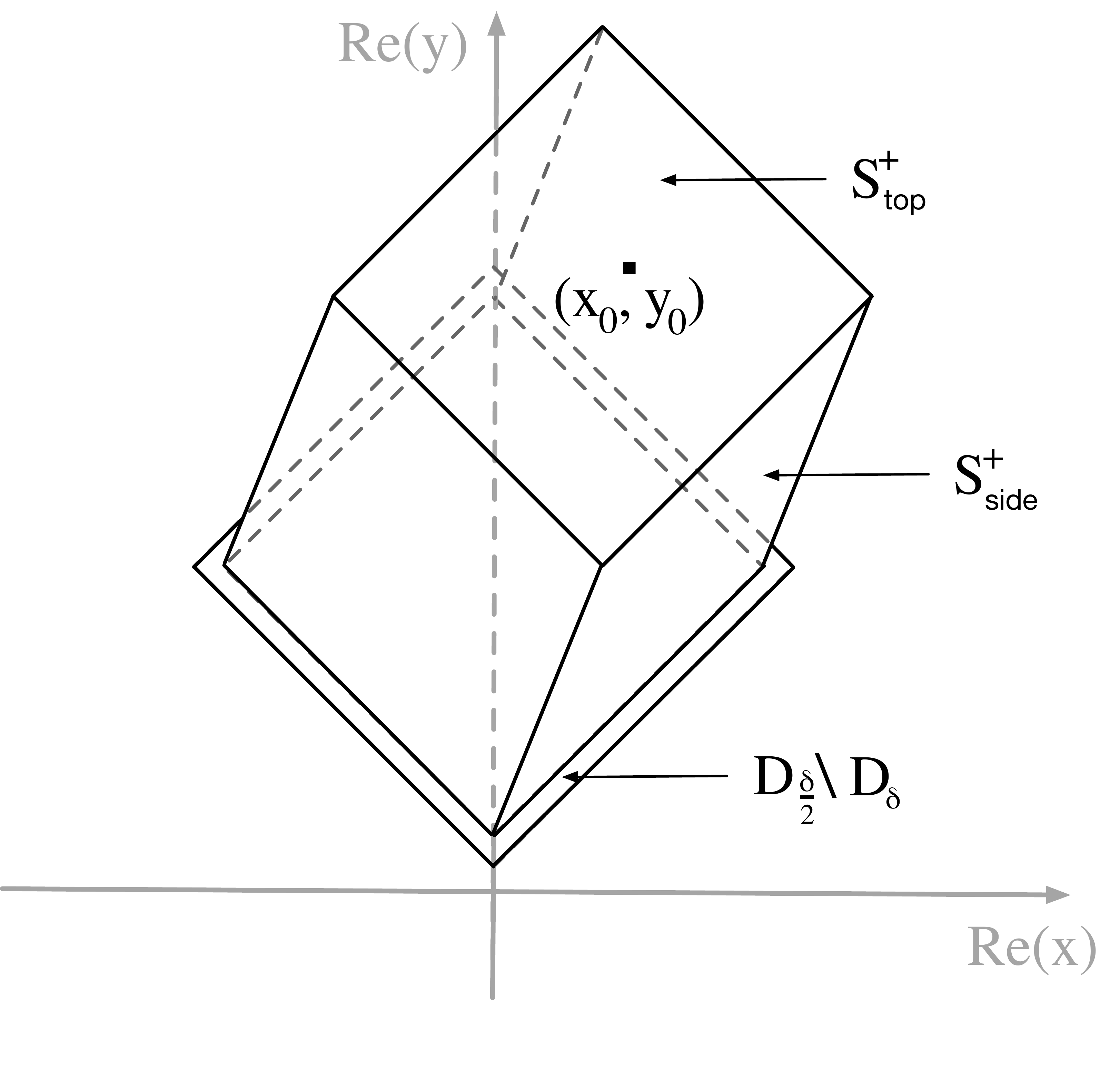}
\caption{The deformed surface $S^+$}
\label{surface} 
\end{figure}

\begin{proposition}\label{max}
On $S^+,$ $\mathrm{Im}V^+$ achieves the only absolute maximum at $(x_0,y_0);$ and on $S^-,$ $\mathrm{Im}V^-$ achieves the only absolute maximum at $(-x_0,y_0).$
\end{proposition}

\begin{proof} By Lemma \ref{convexity}, $\mathrm{Im}V^{\pm}$  is concave down on $S^{\pm}_{\text{top}}.$ Since $(\pm x_0, y_0)$ are respectively the critical points of $\mathrm{Im}V^{\pm},$ they are respectively  the only absolute maximum on $S^{\pm}_{\text{top}}.$

On the side $S^{\pm}_{\text{side}},$ for each $(\theta_1,\theta_2)\in \partial D_{\delta}$ respectively consider the functions $$g^{\pm}_{(\theta_1,\theta_2)}(t)\doteq \mathrm{Im}V^{\pm}(\theta_1\pm it\mathrm{Im}(x_0), \theta_2+it\mathrm{Im}(y_0))$$
on $[0,1].$
We show that $g^{\pm}_{(\theta_1,\theta_2)}(t)<\mathrm{Im}V^{\pm}(\pm x_0,y_0)$ for each  $(\theta_1,\theta_2)\in \partial D_{\delta}$ and $t\in[0,1].$

Indeed, since $(\theta_1,\theta_2)\in\partial D_{\delta},$ $g^{\pm}_{(\theta_1,\theta_2)}(0)=\mathrm{Im}V^{\pm}(\theta_1,\theta_2)<\frac{1}{2}\mathrm{Vol}(\mathrm S^3\setminus K_{4_1})+\epsilon<Vol(M_{K_\theta})=\mathrm{Im}V^{\pm}(\pm x_0,y_0);$ and since $(\theta_1\pm i\mathrm{Im}(x_0),\theta_2+i\mathrm{Im}(y_0))\in S^{\pm}_{\text{top}},$ by the previous step $g^{\pm}_{(\theta_1,\theta_2)}(1) = \mathrm{Im}V^{\pm}(\theta_1\pm i\mathrm{Im}(x_0),\theta_2+i\mathrm{Im}(y_0))<\mathrm{Im}V^{\pm}(\pm x_0,y_0).$
Now by Lemma \ref{convexity}, $g^{\pm}_{(\theta_1,\theta_2)}$ is concave up, and hence
$$g^{\pm}_{(\theta_1,\theta_2)}(t)\leqslant \max\Big\{g^{\pm}_{(\theta_1,\theta_2)}(0), g^{\pm}_{(\theta_1,\theta_2)}(1)\Big\}<\mathrm{Im}V^{\pm}(\pm x_0,y_0).$$ 

By Proposition \ref{bound} and assumption of $\theta,$ on $D_{\frac{\delta}{2}}\setminus D_{\delta},$ $\mathrm{Im}V(x,y)\leqslant \frac{1}{2}\mathrm{Vol}(\mathrm S^3\setminus K_{4_1})+\epsilon<\mathrm{Vol(M_{K_\theta})}=\mathrm{Im}V^\pm(\pm x_0,y_0).$
\end{proof}

\begin{proposition}\label{leading}
 \begin{equation*}
\begin{split}
\int_{D_{\frac{\delta}{2}}}\epsilon(x,y)e^{-xi+\frac{r}{4\pi i}V_r^\pm(x,y)}dxdy=\frac{4\pi}{r}\frac{e^{\frac{r}{4\pi}\Big(\mathrm{Vol}(M_{K_\theta})+i\mathrm{CS}(M_{K_\theta})\Big)}}{\sqrt{-\mathrm{Hess}V^+(x_0,y_0)}}\Big(1+O\Big(\frac{1}{r}\Big)\Big).
\end{split}
\end{equation*}
\end{proposition}

\begin{proof} By analyticity, the integrals remain the same if we deform the domains from $D_{\frac{\delta}{2}}$ to $S^\pm.$ Then by Corollary \ref{c}, $(\pm x_0,y_0)$ are respectively the critical points of $V^\pm.$ By Proposition \ref{max}, $\mathrm{Im}V^\pm$ achieves the only absolute maximum on $S^\pm$ at $(\pm x_0,y_0).$  By Proposition \ref{Vol}, (2), $\mathrm{Hess}V^\pm(\pm x_0,y)\neq 0.$ Finally, to estimate the difference between $V^\pm_r$ and $V^\pm,$ we have 
$$\varphi_r\Big(\pi-x-y-\frac{\pi}{r}\Big)=\varphi_r(\pi-x-y)-\varphi'_r(\pi-x-y)\cdot\frac{\pi}{r}+O\Big(\frac{1}{r^2}\Big)$$
and
$$\varphi_r\Big(y-x+\frac{\pi}{r}\Big)=\varphi_r(y-x)+\varphi'_r(y-x)\cdot\frac{\pi}{r}+O\Big(\frac{1}{r^2}\Big).$$
Then by Lemma \ref{converge}, in $\big\{(x,y)\in \overline{D_{\mathbb C,\delta}}\ \big|\ |\mathrm{Im}x| < L, |\mathrm{Im}x| < L\}$ for some sufficiently large $L,$ 
$$V^\pm_r(x,y)=V^\pm(x,y)-\frac{2\pi i\big(\log\big(1-e^{-2i(y+x)}\big)+\log\big(1-e^{2i(y-x)}\big)\big)}{r}+\frac{\upsilon_r(x,y)}{r^2}$$
with $|\upsilon_r(x,y)|$ bounded from above by a constant independent of $r,$ and
$$e^{-xi+\frac{r}{4\pi i}V_r^\pm(x,y)}=e^{-xi-\frac{\log\big(1-e^{-2i(y+x)}\big)}{2}-\frac{\log\big(1-e^{2i(y-x)}\big)}{2}+\frac{r}{4\pi i}\Big(V^\pm(x,y)+\frac{\upsilon_r(x,y)}{r^2}\Big)}.$$

Now let $D_{\mathbf z}=\big\{(x,y)\in \overline{D_{\mathbb C,\delta}}\ \big|\ |\mathrm{Im}x| < L, |\mathrm{Im}x| < L\}$ for some sufficiently large $L,$  $\mathbf a_r= \theta=\Big|2\pi-\frac{4\pi}{r}-\frac{4\pi m_0}{r}\Big|,$ $f_r^{\mathbf a_r}(x,y)=V_r^\pm(x,y),$ $g_r^{\mathbf a_r}(x,y)=e^{-xi-\frac{\log\big(1-e^{-2i(y+x)}\big)}{2}-\frac{\log\big(1-e^{2i(y-x)}\big)}{2}},$ $f^{\mathbf a_r}(x,y)=V^\pm(x,y).$ Then all the conditions of Proposition \ref{saddle} are satisfied.  By the first equation of (\ref{C}), at the critical points $(\pm x_0,y_0),$
$$-xi-\frac{\log\big(1-e^{-2i(y+x)}\big)}{2}-\frac{\log\big(1-e^{2i(y-x)}\big)}{2}=0,$$
and hence $g_r^{\mathbf a_r}(\pm x_0,y_0)=1.$
By Proposition \ref{Vol}, (1), the critical values $$V^\pm(\pm x_0,y_0)=i\big(\mathrm{Vol}(M_{K_\theta})+i\mathrm{CS}(M_{K_\theta})\big)$$ and the result follows.
\end{proof}


\subsection{Estimates of other Fourier coefficients}\label{other}

Let 
$$V^{(k_0,k_1,k_2),\pm}(x,y)=V^\pm(x,y)-\frac{4k_0\pi x}{q}-4k_1\pi x-4k_2 \pi y, $$ and 
$$\hat F^\pm(k_0, k_1, k_2)=\int_{\mathcal D_{\frac{\delta}{2}}}e^{\frac{r}{4\pi i}V^{(k_0,k_1,k_2),\pm}(x,y)}dxdy.$$
By Lemma \ref{converge}, the asymptotics of $\hat F^\pm_r(k_0,k_1,k_2)$ is approximated by that of $\hat F^\pm(k_0,k_1,k_2).$ 
We will then estimate the contribution to $\hat F^\pm(k_0,k_1,k_2)$ of each individual square $D_{\frac{\delta}{2}},$ $D'_{\frac{\delta}{2}}$ and $D_{\frac{\delta}{2}}''.$

\subsubsection{Estimate on $D_{\frac{\delta}{2}}$}

Let
$$\hat F^\pm_D(k_0, k_1, k_2)=\int_{D_{\frac{\delta}{2}}}e^{\frac{r}{4\pi i}V^{(k_0,k_1,k_2),\pm}(x,y)}dxdy.$$

\begin{lemma}\label{k2} For $k_2\neq 0,$ 
$$\hat F^\pm_D(k_0, k_1,k_2)=O\Big(e^{\frac{r}{4\pi}\big(\mathrm{Vol}(M_{K_\theta})-\epsilon\big)}\Big).$$
\end{lemma}

\begin{proof} In $D_{\mathbb C},$ we have
$$0<\arg (1-e^{-2i(y+x)}) <\pi -2(\mathrm{Re}(y)+\mathrm{Re}(x))$$
and
$$2(\mathrm{Re}(y)-\mathrm{Re}(x))-\pi <\arg(1-e^{2i(y-x)})<0.$$

For $k_2>0,$ let $y=\mathrm{Re}(y)+il.$ Then
\begin{equation*}
\begin{split}
\frac{\partial \mathrm{Im}V^{(k_0, k_1,k_2)}}{\partial l}&=4\mathrm{Re}(x)+2\arg (1-e^{-2i(y+x)})+2\arg(1-e^{2i(y-x)})-4k_2\pi\\
&<4x+2(\pi-2(\mathrm{Re}(y)+\mathrm{Re}(x)))+0-4k_2\pi\\
&=2\pi -4\mathrm{Re}(y)-4k_2\pi<-2\pi,
\end{split}
\end{equation*}
where the last inequality comes from that $0<\mathrm{Re}(y)<\frac{\pi}{2}$ and $k_2>0.$ Therefore, pushing the integral domain along the $il$ direction far enough (without changing $\mathrm{Im}(x)$), the imaginary part of $\mathrm{Im}V^{(k_0,k_1,k_2),\pm}$ becomes smaller than the volume. Since $\mathrm{Im}V^{(k_0,k_1,k_2),\pm}$ is already smaller than the volume of $M_{K_\theta}$ on $\partial D_{\delta},$ it becomes even smaller on the side.

For $k_2<0,$ let $y=\mathrm{Re}(y)-il.$ Then
\begin{equation*}
\begin{split}
\frac{\partial \mathrm{Im}V^{(k_0,k_1,k_2),\pm}}{\partial l}&=-4\mathrm{Re}(x)-2\arg (1-e^{-2i(y+x)})-2\arg(1-e^{2i(y-x)})+4k_2\pi\\
&<-4x-0-2(2(\mathrm{Re}(y)-\mathrm{Re}(x))-\pi)+4k_2\pi\\
&=2\pi -4\mathrm{Re}(y)+4k_2\pi<-2\pi,
\end{split}
\end{equation*}
where the last inequality comes from that $0<\mathrm{Re}(y)<\frac{\pi}{2}$ again and $k_2<0.$ Therefore, pushing the integral domain along the $-il$ direction far enough (without changing $\mathrm{Im}(x)$), the imaginary part of $\mathrm{Im}V^{(k_0,k_1,k_2),\pm}$ becomes smaller than the volume of $M_{K_\theta}.$ Since $\mathrm{Im}V^{(k_0,k_1,k_2),\pm}$ is already smaller than the volume of $M_{K_\theta}$ on $\partial D_{\delta},$ it becomes even smaller on the side.
\end{proof}

\begin{lemma}\label{k1} For $(k_0,k_1)$ so that $\frac{k_0}{q}+k_1\neq 0,$ 
$$\hat F^\pm_D(k_0, k_1,0)=O\Big(e^{\frac{r}{4\pi}\big(\mathrm{Vol}(M_{K_\theta})-\epsilon\big)}\Big).$$
\end{lemma}

\begin{proof}  Here we recall that
$S^\pm= S^\pm_{\text{top}}\cup S^\pm_{\text{side}}\cup(D_{\frac{\delta}{2}}\setminus D_{\delta}),$ where
$$S^\pm_{\text{top}}=\Big\{(x,y)\in D_{\mathbb C,\delta}\ |\ (\mathrm{Im}(x),\mathrm{Im}(y))=(\pm\mathrm{Im}(x_0),\mathrm{Im}(y_0))\Big\}$$ and 
$$S^\pm_{\text{side}}=\Big\{ (\theta_1\pm it\mathrm{Im}(x_0), \theta_2+it \mathrm{Im}(y_0)) \ |\ (\theta_1,\theta_2)\in \partial D_{\delta}, t\in[0,1]\Big\}.$$

By Proposition \ref{max},  for any $(x,y)\in S^\pm_{\text{top}}$ we respectively have 
$$\mathrm{Im}V^\pm(x,y)\leqslant\mathrm{Im}V^\pm(\pm x_0,y_0)=\mathrm{Vol}(M_{K_\theta}).$$

By Lemma \ref{n0}, $\mathrm{Im}(x_0)\neq 0.$ We first consider the case that $\mathrm{Im}(x_0)>0.$ If $\frac{k_0}{q}+k_1>0,$ then on $S^+_{\text{top}}$ we have
\begin{equation*}
\begin{split}
\mathrm{Im}V^{(k_0, k_1, 0),+}(x,y)=&\mathrm{Im}V^+(x,y)-\frac{4k_0\pi}{q}\mathrm{Im}(x_0)-4k_1\pi\mathrm{Im}(x_0)\\
<&\mathrm{Im}V^+(x,y)\leqslant \mathrm{Vol}(M_{K_\theta}),
\end{split}
\end{equation*}
and
\begin{equation*}
\begin{split}
\mathrm{Im}V^{(k_0, k_1, 0),-}(x,y)=&\mathrm{Im}V^{(k_0, k_1, 0),+}(x,y)-\frac{2\theta}{q}\mathrm{Im}(x_0)\\
=&\mathrm{Im}V^+(x,y)-\frac{2\theta}{q}\mathrm{Im}(x_0)-\frac{4k_0\pi}{q}\mathrm{Im}(x_0)-4k_1\pi\mathrm{Im}(x_0)\\
<&\mathrm{Im}V^+(x,y) \leqslant \mathrm{Vol}(M_{K_\theta}).
\end{split}
\end{equation*}
If $\frac{k_0}{q}+k_1<0,$ then on $S^-_{\text{top}}$ we have 
\begin{equation*}
\begin{split}
\mathrm{Im}V^{(k_0, k_1, 0),+}(x,y)=&\mathrm{Im}V^{(k_0, k_1, 0),-}(x,y)-\frac{2\theta}{q}\mathrm{Im}(x_0)\\
=&\mathrm{Im}V^-(x,y)-\frac{2\theta}{q}\mathrm{Im}(x_0)+\frac{4k_0\pi}{q}\mathrm{Im}(x_0)+4k_1\pi\mathrm{Im}(x_0)\\
<&\mathrm{Im}V^-(x,y) \leqslant \mathrm{Vol}(M_{K_\theta}),
\end{split}
\end{equation*}
and
\begin{equation*}
\begin{split}
\mathrm{Im}V^{(k_0, k_1, 0),-}(x,y)=&\mathrm{Im}V^-(x,y)+\frac{4k_0\pi}{q}\mathrm{Im}(x_0)+4k_1\pi\mathrm{Im}(x_0)\\
<&\mathrm{Im}V^-(x,y) \leqslant \mathrm{Vol}(M_{K_\theta}).
\end{split}
\end{equation*}

Next we consider the case that  $\mathrm{Im}(x_0)<0.$ Due to the fact that $k_0$ and $k_1$ are integers and $\theta \in (0,2\pi),$ if $\frac{k_0}{q}+k_1>0,$ then we have $\frac{k_0}{q}+k_1\geqslant \frac{1}{q}$ and 
\begin{equation}\label{>}
\frac{4\pi k_0}{q}+4\pi k_1-\frac{2\theta}{q}>0;
\end{equation}
  and if $\frac{k_0}{q}+k_1<0,$ then we have $\frac{k_0}{q}+k_1\leqslant -\frac{1}{q}$ and 
\begin{equation}\label{<}
\frac{4\pi k_0}{q}+4\pi k_1+\frac{2\theta}{q}<0.
\end{equation}
Now for $(k_0,k_1)$ such that $\frac{k_0}{q}+k_1>0,$ we have on $S^-_{\text{top}}$ that 
\begin{equation*}
\begin{split}
\mathrm{Im}V^{(k_0, k_1, 0),+}(x,y)=&\mathrm{Im}V^{(k_0, k_1, 0),-}(x,y)-\frac{2\theta}{q}\mathrm{Im}(x_0)\\
=&\mathrm{Im}V^-(x,y)-\frac{2\theta}{q}\mathrm{Im}(x_0)+\frac{4k_0\pi}{q}\mathrm{Im}(x_0)+4k_1\pi\mathrm{Im}(x_0)\\
<&\mathrm{Im}V^-(x,y) \leqslant \mathrm{Vol}(M_{K_\theta}),
\end{split}
\end{equation*}
where the penultimate inequality comes from (\ref{>}), and
\begin{equation*}
\begin{split}
\mathrm{Im}V^{(k_0, k_1, 0),-}(x,y)=&\mathrm{Im}V^-(x,y)+\frac{4k_0\pi}{q}\mathrm{Im}(x_0)+4k_1\pi\mathrm{Im}(x_0)\\
<&\mathrm{Im}V^-(x,y) \leqslant \mathrm{Vol}(M_{K_\theta}).
\end{split}
\end{equation*}
For $(k_0,k_1)$ such that $\frac{k_0}{q}+k_1<0,$ we have on $S^+_{\text{top}}$ that 
\begin{equation*}
\begin{split}
\mathrm{Im}V^{(k_0, k_1, 0),+}(x,y)=&\mathrm{Im}V^+(x,y)-\frac{4k_0\pi}{q}\mathrm{Im}(x_0)-4k_1\pi\mathrm{Im}(x_0)\\
<&\mathrm{Im}V^+(x,y)\leqslant \mathrm{Vol}(M_{K_\theta}),
\end{split}
\end{equation*}
and
\begin{equation*}
\begin{split}
\mathrm{Im}V^{(k_0, k_1, 0),-}(x,y)=&\mathrm{Im}V^{(k_0, k_1, 0),+}(x,y)-\frac{2\theta}{q}\mathrm{Im}(x_0)\\
=&\mathrm{Im}V^+(x,y)-\frac{2\theta}{q}\mathrm{Im}(x_0)-\frac{4k_0\pi}{q}\mathrm{Im}(x_0)-4k_1\pi\mathrm{Im}(x_0)\\
<&\mathrm{Im}V^+(x,y) \leqslant \mathrm{Vol}(M_{K_\theta})
\end{split}
\end{equation*}
where the penultimate inequality comes from (\ref{<}).

We note that $V^{(k_0,k_1,k_2),\pm}$ differs from $V^\pm$ by a linear function. Therefore, for each $(\theta_1, \theta_2)\in\partial D_{\delta},$ the function 
$$g^\pm_{(\theta_1,\theta_2)}(t)\doteq \mathrm{Im}V^{(k_0, k_1, 0)}(\theta_1\pm it\mathrm{Im}(x_0),\theta_2+it\mathrm{Im}(y_0))$$
is concave up on $[0,1],$ and hence 
$$g^{\pm}_{(\theta_1,\theta_2)}(t)\leqslant \max\Big\{g^{\pm}_{(\theta_1,\theta_2)}(0), g^{\pm}_{(\theta_1,\theta_2)}(1)\Big\}<\mathrm{Vol}(M_{K_\theta}).$$ 

Putting all together, we have: In the $\mathrm{Im}(x_0)>0$ case, if $\frac{k_0}{q}+k_1>0,$ then $\mathrm{Im}V^{(k_0, k_1, 0),\pm}(x,y)<\mathrm{Vol}(M_{K_\theta})$ on $S^+,$ and if  $\frac{k_0}{q}+k_1<0,$ then $\mathrm{Im}V^{(k_0, k_1, 0),\pm}(x,y)<\mathrm{Vol}(M_{K_\theta})$ on $S^-.$ In the $\mathrm{Im}(x_0)<0$ case, if  $\frac{k_0}{q}+k_1>0,$ then $\mathrm{Im}V^{(k_0, k_1, 0),\pm}(x,y)<\mathrm{Vol}(M_{K_\theta})$ on $S^-,$ and  if $\frac{k_0}{q}+k_1<0,$ then $\mathrm{Im}V^{(k_0, k_1, 0),\pm}(x,y)<\mathrm{Vol}(M_{K_\theta})$ on $S^+.$
\end{proof}


\subsubsection{Estimate on $D_{\frac{\delta}{2}}'$}

Let 
$$\hat F^\pm_{D'}(k_0, k_1,k_2)=\int_{D'_\delta}e^{\frac{r}{4\pi i}V^{(k_0,k_1,k_2),\pm}(x,y)}dxdy.$$

\begin{lemma}\label{D'}
 For any triple $(k_0, k_1, k_2),$ 
$$\hat F^\pm_{D'}(k_0, k_1,k_2)=O\Big(e^{\frac{r}{4\pi}\big(\mathrm{Vol}(M_{K_\theta})-\epsilon\big)}\Big).$$
\end{lemma}

\begin{proof} In $D'_{\mathbb C,\delta},$ we have
$$0<\arg (1-e^{-2i(y+x)}) <\pi -2(\mathrm{Re}(y)+\mathrm{Re}(x))$$
and
$$2(\mathrm{Re}(y)-\mathrm{Re}(x))-3\pi <\arg(1-e^{2i(y-x)})<0.$$

For $k_2\geqslant 0,$ let $y=\mathrm{Re}(y)+il.$ Then
\begin{equation*}
\begin{split}
\frac{\partial \mathrm{Im}V^{(k_0,k_1,k_2),\pm}}{\partial l}&=4\mathrm{Re}(x)+2\arg (1-e^{-2i(y+x)})+2\arg(1-e^{2i(y-x)})-4k_2\pi\\
&<4x+2(\pi-2(\mathrm{Re}(y)+\mathrm{Re}(x)))+0-4k_2\pi\\
&=2\pi -4\mathrm{Re}(y)-4k_2\pi<-2\delta,
\end{split}
\end{equation*}
where the last inequality comes from that $\frac{\pi}{2}+\frac{\delta}{2}<\mathrm{Re}(y)<\pi-\frac{\delta}{2}$ and $k_2\geqslant0.$ Therefore, pushing the integral domain along the $il$ direction far enough (without changing $\mathrm{Im}(x)$), the imaginary part of $\mathrm{Im}V^{(k_0,k_1,k_2),\pm}$ becomes smaller than the volume of $M_{K_\theta}.$ Since $\mathrm{Im}V^{(k_0,k_1,k_2),\pm}$ is already smaller than the volume of $M_{K_\theta}$ on $\partial D'_{\delta},$ it becomes even smaller on the side.

For $k_2<0,$ let $y=\mathrm{Re}(y)-il.$ Then
\begin{equation*}
\begin{split}
\frac{\partial \mathrm{Im}V^{(k_0,k_1,k_2),\pm}}{\partial l}&=-4\mathrm{Re}(x)-2\arg (1-e^{-2i(y+x)})-2\arg(1-e^{2i(y-x)})+4k_2\pi\\
&<-4x-0-2(2(\mathrm{Re}(y)-\mathrm{Re}(x))-3\pi)+4k_2\pi\\
&=6\pi -4\mathrm{Re}(y)+4k_2\pi<-2\delta,
\end{split}
\end{equation*}
where the last inequality comes from that $\frac{\pi}{2}+\frac{\delta}{2}<\mathrm{Re}(y)<\pi-\frac{\delta}{2}$ again and $k_2<0.$ Therefore, pushing the integral domain along the $-il$ direction far enough (without changing $\mathrm{Im}(x)$), the imaginary part of $\mathrm{Im}V^{(k_0,k_1,k_2),\pm}$ becomes smaller than the volume of $M_{K_\theta}.$ Since $\mathrm{Im}V^{(k_0,k_1,k_2),\pm}$ is already smaller than the volume of $M_{K_\theta}$ on $\partial D'_{\delta},$ it becomes even smaller on the side.
\end{proof}

\subsubsection{Estimate on $D_{\frac{\delta}{2}}''$}
Let 
$$\hat F^\pm_{D''}(k_0, k_1, k_2)=\int_{D''_{\delta}}e^{\frac{r}{4\pi i}V^{(k_0,k_1,k_2),\pm}(x,y)}dxdy.$$

\begin{lemma}\label{D''}
 For any triple $(k_0, k_1, k_2),$ 
$$\hat F^\pm_{D''}(k_0, k_1,k_2)=O\Big(e^{\frac{r}{4\pi}\big(\mathrm{Vol}(M_{K_\theta})-\epsilon\big)}\Big).$$
\end{lemma}

\begin{proof} In $D''_{\mathbb C,\delta},$ we have
$$0<\arg (1-e^{-2i(y+x)}) <3\pi -2(\mathrm{Re}(y)+\mathrm{Re}(x))$$
and
$$2(\mathrm{Re}(y)-\mathrm{Re}(x))-\pi <\arg(1-e^{2i(y-x)})<0.$$

For $k_2>0,$ let $y=\mathrm{Re}(y)+il.$ Then
\begin{equation*}
\begin{split}
\frac{\partial \mathrm{Im}V^{(k_0,k_1,k_2),\pm}}{\partial l}&=4\mathrm{Re}(x)+2\arg (1-e^{-2i(y+x)})+2\arg(1-e^{2i(y-x)})-4k_2\pi\\
&<4x+2(3\pi-2(\mathrm{Re}(y)+\mathrm{Re}(x)))+0-4k_2\pi\\
&=6\pi -4\mathrm{Re}(y)-4k_2\pi<-2\delta,
\end{split}
\end{equation*}
where the last inequality comes from that $\frac{\pi}{2}+\frac{\delta}{2}<\mathrm{Re}(y)<\pi-\frac{\delta}{2}$ and $k_2>0.$ Therefore, pushing the integral domain along the $il$ direction far enough (without changing $\mathrm{Im}(x)$), the imaginary part of $\mathrm{Im}V^{(k_0,k_1,k_2),\pm}$ becomes smaller than the volume of $M_{K_\theta}.$ Since $\mathrm{Im}V^{(k_0,k_1,k_2),\pm}$ is already smaller than the volume of $M_{K_\theta}$ on $\partial D''_{\delta},$ it becomes even smaller on the side.

For $k_2\leqslant 0,$ let $y=\mathrm{Re}(y)-il.$ Then
\begin{equation*}
\begin{split}
\frac{\partial \mathrm{Im}V^{(k_0,k_1,k_2),\pm}}{\partial l}&=-4\mathrm{Re}(x)-2\arg (1-e^{-2i(y+x)})-2\arg(1-e^{2i(y-x)})+4k_2\pi\\
&<-4x-0-2(2(\mathrm{Re}(y)-\mathrm{Re}(x))-\pi)+4k_2\pi\\
&=2\pi -4\mathrm{Re}(y)+4k_2\pi<-2\delta,
\end{split}
\end{equation*}
where the last inequality comes from that $\frac{\pi}{2}+\frac{\delta}{2}<\mathrm{Re}(y)<\pi-\frac{\delta}{2}$ again and $k_2\leqslant 0.$ Therefore, pushing the integral domain along the $-il$ direction far enough (without changing $\mathrm{Im}(x)$), the imaginary part of $\mathrm{Im}V^{(k_0,k_1,k_2),\pm}$ becomes smaller than the volume of $M_{K_\theta}.$ Since $\mathrm{Im}V^{(k_0,k_1,k_2),\pm}$ is already smaller than the volume of $M_{K_\theta}$ on $\partial D''_{\delta},$ it becomes even smaller on the side.\end{proof}


\subsection{Proof of Theorem {\ref{main}}}\label{proof}

Theorem \ref{main} follows from the following proposition.
\begin{proposition}\label{4.1} 
\begin{enumerate}[(1)]
\item \begin{equation*}
\begin{split}
\hat f^+_r(0,\dots, 0)+&\hat f^-_r(,0,\dots, 0)=\frac{-2i^{-\frac{k-3}{2}}r^{\frac{k+1}{2}}}{\pi\sqrt{q}\sqrt{-\mathrm{Hess}V^+(x_0,y_0)}}e^{\frac{r}{4\pi}\big(\mathrm{Vol}(M_{K_\theta})+i\mathrm{CS}(M_{K_\theta})\big)}\Big(1+O\Big(\frac{1}{\sqrt r}\Big)\Big).
\end{split}
\end{equation*}

\item For $(n_1,\dots,n_{k-1},k_1,k_2)\neq (0,\dots,0),$ 
\begin{equation*}
\begin{split}
|\hat f^\pm_r(n_1,\dots, n_{k-1},k_1,k_2)|\leqslant O\Big(r^{\frac{k}{2}}\Big)e^{\frac{r}{4\pi}\mathrm{Vol}(M_{K_\theta})}.
\end{split}
\end{equation*}
\end{enumerate}
\end{proposition}

\begin{proof} By Lemmas \ref{D'} and \ref{D''}, the contribution of $D'_{\frac{\delta}{2}}$ and $D''_{\frac{\delta}{2}}$ to $\hat f_r(n_1,\dots, n_{k-1}, k_1, k_2)$ is of order $O\big(e^{\frac{r}{4\pi}\big(\mathrm{Vol}(M_{K_\theta})-\epsilon\big)}\big),$ hence is neglectable.

Then (1) follows from Propositions \ref{0} and \ref{leading}. 
 
To see (2), for $(n_1,\dots,n_{k-1},k_1,k_2)\neq (0,\dots,0)$ with $\big(\frac{k_0}{q}+k_1,k_2\big)\neq (0,0),$ by Proposition \ref{other'} (1), Lemmas \ref{k2} and \ref{k1}, and Lemma \ref{converge}, we have that 
$$\hat f^\pm_r(n_1, n_2 ,\dots, n_{k-1}, k_1, k_2)=O\Big(e^{\frac{r}{4\pi}\big(\mathrm{Vol}(M_{K_\theta})-\epsilon\big)}\Big).$$

If $\big(\frac{k_0}{q}+k_1,k_2\big)= (0,0),$ in particular, if $\frac{k_0}{q}+k_1=0,$ then $|k_0|=|qk_1|$ is either $0$ or greater than or equal to  $q.$ Then by Propositions \ref{other'} (2) and \ref{leading}, and Lemma \ref{converge},
$$|\hat f^\pm_r( n_1,n_2,\dots, n_{k-1},k_1,k_2)|\leqslant O\Big(r^{\frac{k}{2}}\Big)e^{\frac{r}{4\pi}\mathrm{Vol}(M_{K_\theta})}.$$
\end{proof}

\begin{proof}[Proof of Theorem \ref{main}]
By Propositions \ref{Poisson} and \ref{4.1}, we have
\begin{equation*}
\begin{split}
&\lim_{r\to\infty}\frac{4\pi}{r} \log \mathrm{RT}_r(M_{K_\theta})\\
=&\lim_{r\to\infty}\frac{4\pi}{r}\bigg(\log \kappa_r +\log\bigg(\sum \hat f^\pm_r(n_1,\dots, n_{k-1},k_1,k_2)+O\Big(e^{\frac{r}{4\pi}\big(\frac{1}{2}\mathrm{Vol}(\mathrm{S}^3\setminus K_{4_1})+\epsilon\big)}\Big)\bigg)\bigg)\\
=&i\Big(3\sum_{i=1}^ka_i+\sigma(L)+2k-2\Big)\pi^2+\mathrm{Vol}(M_{K_\theta})+i\mathrm{CS}(M_{K_\theta})\\
=&\mathrm{Vol}(M_{K_\theta})+i\mathrm{CS}(M_{K_\theta})\quad\text{mod }i\pi^2\mathbb Z.
\end{split}
\end{equation*}
\end{proof}


\subsection{Cone angles satisfying the assumption of Theorem \ref{main}}

\begin{proposition}\label{small} 
\begin{enumerate}[(1)]
\item  If $(p,q)\neq(\pm 1, 0),$ $(0,\pm 1),$ $(\pm 1,\pm 1),$ $(\pm 2,\pm 1),$ $(\pm 3,\pm 1),$ $(\pm 4,\pm 1)$ and $(\pm 5,\pm 1),$ then for any cone angle  $\theta$ less than or equal to $2\pi,$
$$\mathrm{Vol}(M_{K_\theta})>\frac{\mathrm{Vol}(S^3\setminus K_{4_1})}{2}.$$
\item  If $(p,q)=(0,\pm 1),$ $(\pm 1,\pm 1),$ $(\pm 2,\pm 1),$ $(\pm 3,\pm 1),$ $(\pm 4,\pm 1)$ or $(\pm 5,\pm 1),$ then for any cone angle  $\theta$ less than or equal to $\pi,$
$$\mathrm{Vol}(M_{K_\theta})>\frac{\mathrm{Vol}(S^3\setminus K_{4_1})}{2}.$$
\end{enumerate}
\end{proposition}

\begin{proof} By \cite[Proposition 7.2]{WY}, (1) holds for $\theta=2\pi.$ Then by \cite[Corollary 5.4]{H} that $\mathrm{Vol}(M_{K_\theta})$ is decreasing in $\theta,$ (1) holds for all $\theta$ less than $2\pi.$ 

For (2),  by \cite[Chapter 6]{H} have that for $(p,q)=(0,\pm 1),$ $(\pm 1,\pm 1),$ $(\pm 2,\pm 1),$ $(\pm 3,\pm 1),$ $(\pm 4,\pm 1)$ or $(\pm 5,\pm 1)$ and for any $\theta$ less than $2\pi$  there exists a unique hyperbolic cone metric on $M$ with singular locus $K$ with cone angle $\theta,$ with $\mathrm{Vol}(M_{K_0})=\mathrm{Vol}(S^3\setminus K_{4_1})$ and $\mathrm{Vol}(M_{K_{2\pi}})=\mathrm{Vol}(M)\geqslant 0.$ Let $L_\theta$ be the length of $K$ in $M_{K_\theta}.$ Then by the Schl\"afli formula \cite[Theorem 5.2]{H}, 
$$\frac{d^2\mathrm{Vol}(M_{K_\theta})}{d\theta^2}=-\frac{dL_\theta}{d\theta}.$$ Let $m$ and $l$ respectively be the meridian and longitude of  the boundary of the figure-8 complement as in Section \ref{geometry}, then $m$ is isotopic to $K$ in $M,$ and 
\begin{equation*}
\left\{\begin{array}{c}
p\mathrm H(m)+\mathrm H(l)=\theta i\\
\\
\mathrm H(m) = L_\theta.
\end{array}\right.
\end{equation*}
As a consequence, we have 
$$p+\frac{d\mathrm H(l)}{d L_\theta}=i\frac{d\theta}{d L_\theta},$$
and hence 
$$\frac{d L_\theta}{d\theta}=\bigg(\mathrm{Im} \frac{d\mathrm H(l)}{d\mathrm H(m)}\bigg)^{-1}.$$
By \cite[Formula (68)]{NZ}, we have
$$\frac{d\mathrm H(l)}{d\mathrm H(m)}=2\frac{1-2e^{L_\theta}-2e^{-L_\theta}}{\sqrt{e^{2L_\theta}+e^{-2L_\theta}-2e^{L_\theta}-2e^{-L_\theta}+1}},$$
which implies that
$\mathrm{Im} \frac{d\mathrm H(l)}{d\mathrm H(m)}>0,$ and hence $\frac{d L_\theta}{d\theta}>0$ and $\frac{d^2\mathrm{Vol}(M_{K_\theta})}{d\theta^2}=-\frac{dL_\theta}{d\theta}<0.$
As a consequence, $\mathrm{Vol}(M_{K_\theta})$ is strictly concave down in $\theta,$ and 
$$\mathrm{Vol}(M_{K_\pi})>\frac{\mathrm{Vol}(M_{K_0})+\mathrm{Vol}(M_{K_{2\pi}})}{2}\geqslant \frac{\mathrm{Vol}(S^3\setminus K_{4_1})}{2}.$$
Then by the monotonicity, 
$$\mathrm{Vol}(M_{K_\theta})>\frac{\mathrm{Vol}(S^3\setminus K_{4_1})}{2}$$
for all cone angle  $\theta$ less than or equal to $\pi.$
\end{proof}

\begin{remark} We note that the upper bound $\pi$ in (2) of Proposition \ref{small} is not sharp. 
\end{remark}

\noindent
Ka Ho Wong\\
Department of Mathematics\\  Texas A\&M University\\
College Station, TX 77843, USA\\
(daydreamkaho@math.tamu.edu)
\\

\noindent
Tian Yang\\
Department of Mathematics\\  Texas A\&M University\\
College Station, TX 77843, USA\\
(tianyang@math.tamu.edu)


\begin{thebibliography}{99}



\bibitem{BDKY} G. Belletti, R. Detcherry, E. Kalfagianni, and T. Yang, {\em Growth of quantum 6j-symbols and applications to the Volume Conjecture}, to appear in J. Differential Geom. 


\bibitem{BHMV} C. Blanchet, N. Habegger, G. Masbaum and P. Vogel, {\em Three-manifold invariants derived from the Kauffman bracket}, Topology 31 (1992), no. 4, 685--699.

\bibitem{CM} Q. Chen and J. Murakami, {\em Asymptotics of quantum $6j$ symbols}, preprint, arXiv:1706.04887.


\bibitem{CY} Q. Chen and T. Yang, {\em Volume Conjectures for the Reshetikhin-Turaev and the Turaev-Viro Invariants}, Quantum
Topol. 9 (2018), no. 3, 419--460.

\bibitem{CHK} D. Cooper, C. Hodgson and S. Kerckhoff, {\em Three-dimensional orbifolds and cone-manifolds. With a postface by Sadayoshi Kojima}, MSJ Memoirs, 5. Mathematical Society of Japan, Tokyo, 2000. x+170 pp. ISBN: 4-931469-05-1.



\bibitem{DK} R. Detcherry and E. Kalfagianni, {\em Gromov norm and Turaev-Viro invariants of 3-manifolds}, to appear in Ann. Sci. de l'Ecole Normale Sup..

\bibitem{DKY} R. Detcherry, E. Kalfagianni and T. Yang, {\em Turaev-Viro invariants, colored Jones polynomials and volume}, Quantum Topol. 9 (2018), no. 4, 775--813.


\bibitem{F} L. Faddeev, {\em Discrete Heisenberg-Weyl group and modular group}, Lett. Math. Phys. 34 (1995), no. 3, 249--254.

\bibitem{FKV} L. Faddeev, R. Kashaev and A. Volkov, {\em Strongly coupled quantum discrete Liouville theory}, I. Algebraic approach and duality. Comm. Math. Phys. 219 (2001), no. 1, 199--219.



\bibitem{GL} S. Garoufalidis and T. Le,  {\em Asymptotics of the colored Jones function of a knot}, Geom. Topol. 15 (2011), no. 4, 2135--2180.

\bibitem{Ha} K. Habiro. {\em On the colored Jones polynomial of some simple links}, In: Recent Progress Towards the Volume Conjecture, Research Institute for Mathematical Sciences (RIMS) Kokyuroku 1172, September 2000.


\bibitem{H}  C. Hodgson, {\em Degeneration and regeneration of geometric structures on three-manifolds (foliations, Dehn surgery)}, Thesis (Ph.D.) Princeton University. 1986. 175 pp.


\bibitem{HK} C. Hodgson and S. Kerckhoff, {\em Rigidity of hyperbolic cone-manifolds and hyperbolic Dehn surgery}, 
J. Differential Geom. 48 (1998) 1--59.

\bibitem{Ka1} R. Kashaev, {\em A link invariant from quantum dilogarithm}, Modern Phys. Lett. A 10 (1995), no. 19, 1409--1418.

\bibitem{Ka2} R. Kashaev, {\em The hyperbolic volume of knots from the quantum dilogarithm}, Lett. Math. Phys. 39 (1997), no. 3, 269--275.

\bibitem{Ko} S. Kojima, {\em Deformations of hyperbolic 3-cone-manifolds}, J. Differential Geom. 49 (1998), no. 3, 469--516.

\bibitem{L} D. London, {\em A note on matrices with positive definite real part}, Proc. Amer. Math. Soc. 82 (1981), no. 3, 322--324.

\bibitem{Li} W. Lickorish, {\em The skein method for three-manifold invariants},  J. Knot Theory Ramifications 2 (1993), no. 2, 171--194. 



\bibitem{Ma} G. Masbaum {\em Skein-theoretical derivation of some formulas of Habiro}, Algebr. Geom. Topol., 3, no. 1 (2003), 537--556.


\bibitem{NZ} W. Neumann and D. Zagier, {\em Volumes of hyperbolic three-manifolds}, Topology 24 (1985), no. 3, 307--€"332.

\bibitem{O} T. Ohtsuki, {\em On the asymptotic expansion
of the Kashaev invariant of the $5_2$ knot}, Quantum Topol. 7 (2016), no. 4, 669€--735. 


\bibitem{O2} T. Ohtsuki, {\em On the asymptotic expansion of the quantum $SU(2)$ invariant at $q=\exp(4\pi\sqrt{-1}/N)$ for closed hyperbolic 3-manifolds obtained by integral surgery along the figure-eight knot}, Algebr. Geom. Topol. 18 (2018), no. 7, 4187--4274. 


\bibitem{RT90} N. Reshetikhin and V. Turaev, {\em Ribbon graphs and their invariants derived from quantum groups}, Comm. Math. Phys. 127 (1990), no. 1, 1--26.

\bibitem{RT91}N. Reshetikhin and V. Turaev, {\em Invariants of $3$-manifolds via link polynomials and quantum groups}, Invent. Math. 103 (1991), no. 3, 547--597.

\bibitem{R} D. Rolfsen, {\em Knots and links}, 2nd printing with corrections, Mathematics Lecture
Series 7, Publish or Perish, Inc. (1990).

\bibitem{SS} E. Stein and R. Shakarchi, {\em Fourier analysis}, An introduction. Princeton Lectures in Analysis, 1. Princeton University Press, Princeton, NJ, 2003. xvi+311 pp. ISBN: 0-691-11384-X.

\bibitem{T} W. Thurston, {\em The geometry and topology of $3$-manifolds}, Princeton Univ. Math.
Dept. (1978). Available from http://msri.org/publications/books/gt3m/.


\bibitem{WY} K. H. Wong and T. Yang, {\em On the Volume Conjecture for hyperbolic Dehn-filled $3$-manifolds along the figure-eight knot}, arXiv:2003.10053.

\bibitem{WY2} K. H. Wong and T. Yang, {\em Relative Reshetikhin-Turaev invariants, hyperbolic cone metrics and discrete Fourier transforms I}, arXiv: 2008.05045.

\bibitem{W} E. Witten, {\em Quantum field theory and the Jones polynomial}, Comm. Math. Phys. 121 (3): 351--399.

\bibitem{Y} T. Yoshida, {\em The $\eta$-invariant of hyperbolic $3$-manifolds}, Invent. Math. 81 (1985), no. 3, 473--514.


\bibitem{Z} D. Zagier, {\em The dilogarithm function}, Frontiers in number theory, physics, and geometry. II, 3–65, Springer, Berlin, 2007.



\end{thebibliography}
\end{document}